\documentclass[10pt,reqno]{amsart}
\usepackage[utf8]{inputenc}
\usepackage{graphicx, cite, mathtools}
\usepackage{amsmath,amsthm,amssymb,tocvsec2,pdflscape,cite}
\usepackage{latexsym}
\usepackage{esint}
\usepackage{mathtools}
\usepackage{color,soul}
\usepackage{enumerate}
\usepackage[T1]{fontenc}
\usepackage[mathscr]{euscript}

\usepackage[colorlinks=true, allcolors=blue, breaklinks=true]{hyperref}

\allowdisplaybreaks
\numberwithin{equation}{section}
\usepackage{cleveref}
\newtheorem{theorem}{Theorem}[section]

\newtheorem{lemma}[theorem]{Lemma}
\newtheorem{proposition}[theorem]{Proposition}

\theoremstyle{definition}

\title[On Hardy-type inequalities]{On improved Hardy-type inequalities in Log-Weighted Sobolev Spaces }

\author{Sourav Ghosh}
\address{Department of Mathematics, Indian Institute of Science, Bangalore, India}
\email{souravghosh1@iisc.ac.in}

\subjclass[2020]{Primary 46E35; Secondary 46E30, 26D10, 26D15}
\keywords{Hardy inequalities, series expansion, refined Hardy inequalities, Leray-Trudinger inequalities, weighted Sobolev spaces, logarithmic weights, weighted Leray energy.}
\begin{document}
\begin{abstract}
     We prove sharp Hardy-type inequalities in Log-weighted Sobolev spaces. Due to the non-attainability of the sharp constant, we improve these inequalities by adding suitable remainder terms consisting of optimal weights and sharp constants. Consequently, this gives an infinite series improvement in Log-weighted Sobolev spaces. Next, by using these improved Hardy-type inequalities, we also prove an optimal Leray-Trudinger-type improvement. Furthermore, we derive interpolation between Hardy-type inequalities and Leray-Trudinger-type inequalities.
\end{abstract}
\maketitle
\tableofcontents
\section{Introduction and Statement of Main Results}
Hardy-type inequalities play a crucial role in the study of Harmonic Analysis, Partial Differential Equations, Geometric Analysis, Spectral Theory, and Mathematical Physics. Let $\Omega$ be a bounded domain in $\mathbb{R}^N$, and  set $R_{\Omega} = \sup_{x \in \Omega} \lvert x \rvert$. Then, for all $u\in C_c^{\infty}(\Omega)$ the classical Hardy inequality for $1<p<N$ states as,
\begin{align} \label{Hardy1}
   I_{p,N,\Omega}[u] \coloneqq \int_{\Omega} \lvert \nabla u \rvert^p \, dx -\left(\frac{N-p}{p}\right)^p \int_{\Omega} \frac{\lvert u \rvert^p}{\lvert x \rvert^p} \, dx \ge 0,
\end{align}
whereas for $p=N$ the Hardy inequality takes the form:
\begin{align} \label{Hardy2}
   I_{N,\Omega}[u] \coloneqq \int_{\Omega} \lvert \nabla u \rvert^N \, dx - \left(\frac{N-1}{N}\right)^N \int_{\Omega} \frac{\lvert u \rvert^N}{\lvert x \rvert^N} X_1^N\left(\frac{\lvert x \rvert}{R_{\Omega}}\right)\, dx \ge 0,
\end{align}
where we denote $X_1(r): = \left(\ln (e/r)\right)^{-1},$ for $0\leq r\leq1$ and we also set iterated logarithmic functions as $X_j(r) = X_1\left(X_{j-1}(r)\right)$ for $j=2,3, \cdots.$
The constants appearing in \eqref{Hardy1} and \eqref{Hardy2} are optimal and never attained. For the proof of the inequality \eqref{Hardy1} \emph{i.e.} in case of  $1 < p <N$, one can refer to \cite{HardyLP, Kufner-Hardy, GPerel}. For $p=N=2,$ \eqref{Hardy2} was proved by Leray  \cite[Inequality (5), Chapter III]{leray1933etude} and for general $p=N,$ see \cite{AdiCR, ASan, Barbatis}. We call $I_{N,\Omega}$ the \emph{Leray energy} functional. The lack of attainability of the best constants has motivated the search for refined inequalities involving remainder terms consisting of singular weights and logarithmic corrections.
\par For $p=2,$ an improved version of \eqref{Hardy1} was established by Brezis and Vazquez in \cite{BrezisVazquez}. They proved the following improved Hardy inequality: for all $u \in C_c^{\infty}(\Omega),$
\begin{align} \label{BV1}
    I_{2,N,\Omega}[u] \ge \lambda_N \int_{\Omega} u^2 \, dx, \quad \text{for some positive constant} \hspace{3pt} \lambda_N.
\end{align}
In that same paper, they proposed a question (see Problem $2$ of Section $8$ in \cite{BrezisVazquez} ) that if there are only two terms of a series in \eqref{BV1}.
This problem has been settled. In fact, for general $p-$Hardy an infinite series of remainder terms was established by Barbatis, Filippas, and Tertikas
 (see \cite[Theorem A, Theorem B]{BarbatisSeries})
 : for $1<p<N,$ there exists a positive constant $R_0 = R_0(p,N) \ge R_{\Omega}$ such that, for any $R \ge R_0$ and for all $u \in C^{\infty}_c(\Omega \setminus \{0\})$ we have,
\begin{align} \label{unweighted series expansion1}
    I_{p,N,\Omega}[u] \ge  \frac{p-1}{2p} \left(\frac{N-p}{p}\right)^{p-2} \sum_{j=1}^{\infty} \int_{\Omega} \frac{\lvert u \rvert^p}{\lvert x \rvert^p}  X_1^2\left(\frac{\lvert x \rvert}{R}\right) \cdots X_j^2 \left(\frac{\lvert x \rvert}{R}\right),
\end{align}
and for $p=N$, we have,
\begin{align} \label{unweighted series expansion2}
 I_{N,\Omega}[u]  \ge \frac{1}{2} \left(\frac{N-1}{N}\right)^{N-1} \sum_{j=2}^{\infty} \int_{\Omega} \frac{\lvert u \rvert^N}{\lvert x \rvert^N} X_1^N\left(\frac{\lvert x \rvert}{R_{\Omega}}\right)  X_2^2\left(\frac{\lvert x \rvert}{R_{\Omega}}\right) \cdots X_j^2\left(\frac{\lvert x \rvert}{R_{\Omega}}\right)
\end{align}
where $I_{p,N,\Omega}[u]$ and $ I_{N,\Omega}[u]$ are defined as in \eqref{Hardy1} and \eqref{Hardy2} respectively.
For various improvements in this direction, one can refer to  \cite{ChaudhuriNirmalendu,AdiCR,ASan, GazzolaGrunauMitidieri,FT}.
\par For $p=N=2,$ by the ground state representation (see \cite[Proposition 2.6]{PsaradakisJFA}), for all $u \in C_c^1\left(\Omega\setminus \lbrace 0\rbrace\right)$ with $v= u X_1^{1/2}$ we have, $I_{2,\Omega}[u] = \int_{\Omega} \lvert \nabla v \rvert^2 X_1^{-1} \left(\lvert x \rvert / R_{\Omega}\right) \, dx $. By setting $ J_{N,\Omega} [v] \coloneqq \int_{\Omega} \lvert \nabla v \rvert^N X_1^{-N+1}\left(\lvert x \rvert / R_{\Omega}\right) \,dx, $  from \eqref{unweighted series expansion2} we can obtain,
\begin{align} \label{weighted hardy X_1^-}
     J_{2,\Omega} [v] \ge \frac{1}{4} \sum_{j=2}^{\infty} \int_{\Omega} \frac{v^2}{\lvert x \rvert^2} X_1\left(\frac{\lvert x \rvert}{R_{\Omega}}\right)  X_2^2\left(\frac{\lvert x \rvert}{R_{\Omega}}\right) \cdots X_j^2\left(\frac{\lvert x \rvert}{R_{\Omega}}\right).
\end{align}
 \par We are interested in the higher dimensional analogue \emph{i.e.} when the left-hand side of \eqref{weighted hardy X_1^-} is replaced by $J_{N,\Omega} [v]$, what will be the optimal expansion in the right-hand side. The non-equivalence of $I_{N,\Omega}[u]$ and $J_{N,\Omega} [v]$ on $C_c^{\infty}(\Omega \setminus \{0\})$ (see \cite[Lemma A$.1$]{adimurthi2026leraytrudingertypeexponentialintegrability}) prevents a direct extension of \eqref{weighted hardy X_1^-} to the higher dimensions, so the higher dimensional analogue does not follow immediately from \eqref{unweighted series expansion2}.
\par We write $X_j$ to mean $X_j\left(\lvert x \rvert/R_{\Omega}\right)$ for $j=1,2,\cdots$. First, let us set,
\begin{align} \label{finite series}
  \notag  I_m[u] \coloneqq \int_{\Omega} \lvert \nabla u \rvert^N X_1^{-N+1} \, dx &- \left(\frac{N-1}{N}\right)^N \int_{\Omega} \frac{\lvert u \rvert ^N}{\lvert x \rvert^N} X_2^N X_1 \, dx\\
    &\hspace{-30pt}- \frac{1}{2}\left(\frac{N-1}{N}\right)^{N-1} \sum_{j=3}^m \int_{\Omega} \frac{\lvert u \rvert ^N}{\lvert x \rvert^N} X_2^N X_1 X_3^2 \cdots X_j^2 \, dx,
\end{align}
where $m=3,4,\cdots.$  Let us state our first result as follows.
\begin{theorem} \label{series expansion}
  Let $\Omega$ be a bounded domain containing the origin. Then, for all $u \in C^{\infty}_c(\Omega)$, we have the following inequality,
    \begin{align} \label{series expansion eqn}
     \notag  \int_{\Omega} \lvert \nabla u \rvert^N X_1^{-N+1} \, dx &- \left(\frac{N-1}{N}\right)^N \int_{\Omega} \frac{\lvert u \rvert ^N}{\lvert x \rvert^N} X_2^N X_1 \, dx  \\
   &\ge \frac{1}{2}\left(\frac{N-1}{N}\right)^{N-1} \sum_{j=3}^{\infty} \int_{\Omega} \frac{\lvert u \rvert ^N}{\lvert x \rvert^N} X_2^N X_1 X_3^2 \cdots X_j^2 \, dx .
    \end{align}
Moreover, the constants corresponding to the Hardy-type inequality and the $m-$ improved Hardy-type inequalities are optimal \emph{i.e.}
    \begin{align} \label{sharp const1}
         \left(\frac{N-1}{N}\right)^N = \inf_{u \in C^{\infty}_c(\Omega)} \frac{ \int_{\Omega} \lvert \nabla u \rvert^N X_1^{-N+1} \, dx}{\int_{\Omega} \frac{\lvert u \rvert ^N}{\lvert x \rvert^N} X_2^{N}X_1 \, dx},
    \end{align}
    and for each $m=3,4 \cdots$, also we have,
    \begin{align} \label{sharp const2}
        \frac{1}{2}\left(\frac{N-1}{N}\right)^{N-1}  = \inf_{u \in C^{\infty}_c(\Omega)} \frac{I_{m-1}[u]}{\int_{\Omega} \frac{\lvert u \rvert ^N}{\lvert x \rvert^N} X_2^N X_1 X_3^2 \cdots X_{m-1}^2 \, dx}.
    \end{align}
\end{theorem}
The above Theorem generalizes \eqref{weighted hardy X_1^-} in any arbitrary dimension $N.$ Note that, up to finitely many terms in the unweighted case \eqref{unweighted series expansion2}, the exponent $2$ of $X_2$ is optimal  (see \cite[Theorem B]{BarbatisSeries}) while in the weighted case \eqref{finite series}, the exponent $N$ of $X_2$ is optimal \emph{i.e.} the exponent can not be decreased further, see \eqref{sharp const1} and \eqref{sharp const2}.
\par A similar phenomenon was observed by Adimurthi, Mallick and the author \cite{adimurthi2026leraytrudingertypeexponentialintegrability} when deriving the exponential integrability in Log-weighted Sobolev spaces, particularly regarding the exponent of $X_2$ when considering the Leray energy $I_{N,\Omega}[u]$ and the weighted energy $J_{N,\Omega}[uX_1^{1-1/N}]$ respectively, see \cite[Theorem $1.1$]{adimurthi2026leraytrudingertypeexponentialintegrability} and the discussion following it. The exponential integrability, up to a logarithmic correction, was  first appeared in \cite{PsaradakisJFA} due to Psaradakis and Spector. Their result states as follows: Let $\Omega$ is a bounded domain containing the origin, then for any $\varepsilon>0$, there exists a positive constant $\alpha \equiv \alpha\left(N, \varepsilon\right)$ such that
\begin{align}\label{leray-trudinger}
 \sup_{\left \lbrace u \in W^{1,N}_{0}\left(\Omega\right) : I_{N,\Omega}\left[u\right] \le 1\right\rbrace} \int_{\Omega} e^{\alpha\left(\lvert u \rvert X_1^{\varepsilon}\right)^{N'}} \, dx < \infty.
\end{align}
 Mallick and Tintarev \cite{TM} proved an improved version of \eqref{leray-trudinger}, in which $X_1^{\varepsilon}$ is replaced by $X_2^{2/N}$. They also proved that \eqref{leray-trudinger} fails if $X_1^{\varepsilon}$ is replaced by $X_2^{\kappa}$ for any $\kappa<1/N$.
Recently Di Blasio, Pisante and Psaradakis in \cite{di2024optimal} established the following optimal inequality: for some positive constant $c(N),$ we have,
\begin{align}\label{op-leray-trudinger}
  \sup_{\left \lbrace u \in W^{1,N}_{0}\left(\Omega\right) : I_{N,\Omega}\left[u\right] \le 1\right\rbrace} \fint_{\Omega} e^{\alpha\left(\lvert u \rvert X_2^{\frac{1}{N}}\right)^{N'}} \, dx < \infty, \quad \text{for all} \quad \alpha < c(N).
\end{align}
 \par Next, we will generalize the Leray-Trudinger inequality \cite{di2024optimal, TM, PsaradakisJFA} in the weighted setting which concerns the exponential integrability results in \cite{adimurthi2026leraytrudingertypeexponentialintegrability}.
Let us first define the weighted Leray energy as follows:
\begin{align} \label{weighted Leray energy1}
     \Tilde{I}_N[u] \coloneqq  \int_{\Omega} \lvert \nabla u \rvert^N X_1^{-N+1} \, dx - \left(\frac{N-1}{N}\right)^N \int_{\Omega} \frac{\lvert u \rvert ^N}{\lvert x \rvert^N} X_2^N X_1 \, dx.
\end{align}
We prove the following weighted Leray-Trudinger inequality.
 \begin{theorem} \label{Leray-weighted2}
  Let $\Omega$ be a bounded domain containing the origin. Then, there exists a positive constant $ C_{N}$ such that we have,
    \begin{align} \label{weighted leray-moser}
        \sup_{ \left \lbrace u \in C_c^{\infty}\left(\Omega\right) : \Tilde{I}_N[u] \le 1 \right \rbrace } \int_{\Omega} e^{\alpha \lvert u \rvert^{N'} X_1^{-1} X_2 X_3^{\frac{1}{N-1}}} \, dx < \infty, \mbox{ for all } \alpha < C_N.
    \end{align}
    Moreover, the exponent $1/(N-1)$ of $X_3$ is optimal in the sense that if we take any smaller exponent of $X_3$ then the result is false.
\end{theorem}
 Let us now deal with general Log-weights. Log-weighted Sobolev spaces were initially introduced by Calanchi and Ruf \cite{RufN, RufClanchi2015}  to study the exponential integrability for radial functions, and for non-radial functions it was extended by Adimurthi, Mallick,  and the author \cite{adimurthi2026leraytrudingertypeexponentialintegrability}. Let $w \in L^1_{loc}(\Omega)$ be a non negative function, we define the weighted Sobolev spaces  $W^{1,p}_{0}\left(\Omega, w\right)$ as the closure of  $C^1_c(\Omega)$ under the norm 
 \begin{align}\label{weighted norms}
 \left \lVert \nabla u \right\rVert_{p, w, \Omega} := \left(\int_{\Omega} \left\lvert \nabla u\right\rvert^p w\ dx \right)^{\frac{1}{p}}.
 \end{align}
For a general weighted Sobolev space theory one can refer to \cite{Heinonen, Kufner1, Kufner-Hardy-weight, Kufner-Hardy}. We consider 
the following weight:
for $\beta \in \mathbb{R},$ we set
\begin{align} \label{weights w}
    w_{\beta}(x) &= X_{1}^{-\beta(N-1)}\left(\frac{\lvert x \rvert}{R_{\Omega}}\right),
\end{align}
for $x\in B\left(0, R_{\Omega}\right)$. We will write $\left \lVert \nabla u \right\rVert_{N, w_{\beta}} $ to mean $\left \lVert \nabla u \right\rVert_{N, w_{\beta}, \Omega}$.
\par We investigate the optimal Hardy inequalities in $W^{1,N}_0(\Omega, w_{\beta})$ for $\beta \neq 1.$ The following Hardy type inequality is known  (see Lemma \ref{psara}): Let $\Omega$ be a bounded domain in $\mathbb{R}^N$ then for $\gamma \in \mathbb{R}$ with $\gamma \neq 1$ and for all $u \in C^{\infty}_c\left(\Omega \setminus \{0\} \right)$ we have,
\begin{align} \label{weighted Hardy}
     \int_{\Omega} \lvert \nabla u \rvert^N & X_1^{-\gamma(N-1)} \, dx \ge   \left \lvert \frac{1-\gamma}{N'}\right \rvert^N \int_{\Omega} \frac{\lvert u \rvert ^N}{\lvert x \rvert^N} X_1^{N - (N-1) \gamma} \, dx,
\end{align}
where $N'=N/(N-1)$ denotes the H\"older conjugate of $N.$
The inequality above \eqref{weighted Hardy} is very useful for deriving exponential integrability results, for instance, see \cite{PsaradakisJFA, TM, adimurthi2026leraytrudingertypeexponentialintegrability}. Note that, the case $\gamma =1$ has been covered earlier in \Cref{series expansion}. In this article, we have proved that the constant that appears in \eqref{weighted Hardy} is optimal (see Proposition \ref{sharp hardy type1}). For $\beta <1,$ we have also proved the optimal series expansion of remainder terms analogous to \Cref{series expansion}. The case $\beta < 0,$ can be recovered from \Cref{p=k series expansion}. Let $\beta \in (0,1)$ and for $m=2,3, \cdots$ we set,
\begin{align} \label{Improved beta hardy}
       \notag  I_{m,\beta}[u] &\coloneqq \int_{\Omega} \lvert \nabla u \rvert^N X_1^{-\beta(N-1)} \, dx - \left(\frac{1-\beta}{N'}\right)^N \int_{\Omega} \frac{\lvert u \rvert ^N}{\lvert x \rvert^N} X_1^{N - (N-1) \beta} \, dx\\
      & \qquad  - \frac{1}{2(1-\beta)}\left(\frac{1-\beta}{N'}\right)^{N-1} \sum_{j=2}^{m} \int_{\Omega} \frac{\lvert u \rvert ^N}{\lvert x \rvert^N} X_1^{N - (N-1) \beta}X_2^2 \cdots X_j^2 \, dx.
     \end{align}
We prove the following optimal series expansion result.
\begin{theorem} \label{beta series expansion}
    Let $\Omega$ be a bounded domain in $\mathbb{R}^N$ containing the origin and $\beta \in (0,1)$. Then for all $u \in W^{1,N}_{0}\left(\Omega, w_{\beta}\right)$ there holds:
    \begin{align}
       \notag \int_{\Omega} \lvert \nabla u \rvert^N & X_1^{-\beta(N-1)} \, dx -  \left(\frac{1-\beta}{N'}\right)^N \int_{\Omega} \frac{\lvert u \rvert ^N}{\lvert x \rvert^N} X_1^{N - (N-1) \beta} \, dx\\
        &\ge \frac{1}{2(1-\beta)}\left(\frac{1-\beta}{N'}\right)^{N-1} \sum_{j=2}^{\infty} \int_{\Omega} \frac{\lvert u \rvert ^N}{\lvert x \rvert^N} X_1^{N - (N-1) \beta}X_2^2 \cdots X_j^2\, dx
        \label{beta series}.
    \end{align}
  Moreover, the constants corresponding to the Hardy-type inequality and the $m-$ improved Hardy-type inequalities are optimal \emph{i.e.} 
    \begin{align}
        \left(\frac{1-\beta}{N'}\right)^N = \inf_{u \in W^{1,N}_{0}\left(\Omega , w_{\beta}\right)} \frac{ \int_{\Omega} \lvert \nabla u \rvert^N X_1^{-\beta(N-1)} \, dx}{\int_{\Omega} \frac{\lvert u \rvert ^N}{\lvert x \rvert^N} X_1^{N - (N-1) \beta} \, dx}
    \end{align}
    and for each $m=2,3, \cdots$, also we have,
    \begin{align}
       \frac{1}{2(1-\beta)}\left(\frac{1-\beta}{N'}\right)^{N-1} = \inf_{u \in W^{1,N}_{0}\left(\Omega , w_{\beta}\right)} \frac{I_{m-1, \beta}[u]}{\int_{\Omega} \frac{\lvert u \rvert ^N}{\lvert x \rvert^N} X_1^{N - (N-1) \beta}X_2^2 \cdots X_m^2\, dx}.
    \end{align}
\end{theorem}
Now, for $\beta \neq 1,$ we set
\begin{align} \label{beta weighted leray energy}
        I_{N,w_{\beta}}[u] \coloneqq \int_{\Omega} \lvert \nabla u \rvert^N w_{\beta}\, dx - \left \lvert \frac{1-\beta}{N'}\right \rvert^N \int_{\Omega} \frac{\lvert u \rvert ^N}{\lvert x \rvert^N} w_{\beta}^{1- \frac{N'}{\beta}} \, dx.
\end{align}
Analogously to \Cref{Leray-weighted2}, we  prove the following Leray-Trudinger-type result.
\begin{theorem} \label{Leray weighted1}
 Let $\Omega$ be a bounded domain containing the origin and $\beta \in \mathbb{R}$ such that $\beta \neq 1$. Then, there exists a positive constant $ C_{N,\beta}$ such that we have,
    \begin{align*} 
        \sup_{ \left \lbrace u \in W^{1,N}_0\left(\Omega \setminus \{0\}, w_{\beta}\right) : I_{N,w_{\beta}}[u] \le 1 \right \rbrace } \int_{\Omega} e^{\alpha \lvert u \rvert^{N'} X_1^{-\beta} X_2^{\frac{1}{N-1}}} \, dx < \infty,    \mbox{ for all } \alpha <  C_{N,\beta}.
    \end{align*}
    Moreover, for $\beta \in (-\infty,1)$ the exponent $1/(N-1)$ of $X_2$ is optimal in the sense that if we take any smaller exponent of $X_2$ then the result is false.
\end{theorem}
 Finally, for $\beta < 0,$ we prove series expansion result in a more general setting \emph{i.e.} instead of considering distance from a point, we consider distance from some piecewise smooth closed surface $S$  of codimension $k$, for $1 \le k \le N$. We use some convention: for $k=1,$ we take $S=\partial \Omega,$ and for $k=N$ we take $S=\{0\}.$
 \par We set $d(x) = \text{dist}(x,S).$ Note that $d(x)$ is Lipschitz continuous and $\lvert \nabla d \rvert =1$ a.e.. When $S= \partial \Omega$, Brezis and Marcus \cite{BrezisMarcus} proved that,
    \begin{align} 
 \int_{\Omega} \lvert \nabla u \rvert^2 \, dx &\ge \frac{1}{4} \int_{\Omega} u^2 \left(\frac{1}{d^2} + \frac{1}{\text{diam}^2 \Omega}\right) \, dx, \quad \text{ and } \label{BM1} \\
    \int_{\Omega} \lvert \nabla u \rvert^2 \, dx &\ge \frac{1}{4} \int_{\Omega} \frac{u^2}{d^2} \left(1 + X_1^2\left(\frac{d}{\text{diam}\Omega}\right)\right) \label{BM2}
    \end{align}
    holds for all $u \in C_c^{\infty}(\Omega)$, where $\Omega$ is a convex bounded domain.
\par For general $d,$ we assume the following:
\begin{align} 
   &\mbox{for }p \neq k, \quad \Delta_p d^{\frac{p-k}{p-1}} \le 0, \quad \text{in} \quad \Omega \setminus S, \quad \text{and,} \label{p<k assumption}\\ 
   &\mbox{for } p=k, \quad \Delta_p \left(-\ln d\right) \le 0, \quad \text{in} \quad \Omega \setminus S, \label{p=k assumption}
\end{align}
holds in the weak sense, where $\Delta_p v = \operatorname{div} \left( \lvert \nabla v \rvert^{p-2} \nabla v\right).$ One can refer to \cite{Barbatis} for examples of $S$ and discussions on conditions \eqref{p<k assumption} and \eqref{p=k assumption}. We also assume that 
\begin{align}
    R_{d,\Omega} \coloneqq \sup_{x \in \Omega} d(x) < \infty,
\end{align} and observe that for $d(x)= \lvert x \rvert,$ we have $R_{d,\Omega} = R_{\Omega}.$ In \cite{BarbatisSeries}, by assuming \eqref{p<k assumption} and \eqref{p=k assumption} they showed that: for $p \neq k,$ there exists a positive constant $R_0 = R_0(p,k) \ge R_{d,\Omega}$ such that, for any $R \ge R_0$ and for all $u \in C^{\infty}_c(\Omega \setminus S)$ we have,
\begin{align} \label{general unweighted series expansion1}
  \notag  \int_{\Omega} \lvert \nabla u \rvert^p - &\left \lvert\frac{k-p}{p}\right \rvert^{p} \int_{\Omega} \frac{\lvert u \rvert^p}{d^p}\\
  & \hspace{30pt} \ge \frac{p-1}{2p} \left \lvert \frac{k-p}{p}\right \rvert^{p-2} \sum_{j=1}^{\infty} \int_{\Omega} \frac{\lvert u \rvert^p}{d^p}  X_1^2\left(\frac{\lvert x \rvert}{R}\right) \cdots X_j^2 \left(\frac{\lvert x \rvert}{R}\right),
\end{align}
and for $p=k$, we have,
\begin{align} \label{general unweighted series expansion2}
   \notag  \int_{\Omega} &\lvert \nabla u \rvert^k -  \left(\frac{k-1}{k}\right)^k \int_{\Omega} \frac{\lvert u \rvert^k}{d^k} X_1^k\left(\frac{d}{R_{d,\Omega}}\right) \\
  & \hspace{5pt} \ge \frac{1}{2} \left(\frac{k-1}{k}\right)^{k-1} \sum_{j=2}^{\infty} \int_{\Omega} \frac{\lvert u \rvert^k}{d^k} X_1^k\left(\frac{d}{R_{d,\Omega}}\right)  X_2^2\left(\frac{d}{R_{d,\Omega}}\right) \cdots X_j^2\left(\frac{d}{R_{d,\Omega}}\right).
\end{align}
Let us now consider weights of the form $X_1^{\lambda}\left(d/R_{d,\Omega}\right)$ for $\lambda>0$, where $d$ is same as we mentioned earlier together with the assumptions \eqref{p<k assumption} and \eqref{p=k assumption}. If the context is clear, we will write $X_j$ to mean $X_j\left(d/R_{d,\Omega}\right)$ for $j=1,2,\cdots.$ For $\lambda > 0$ and for $m=2,3, \cdots$ we set,
\begin{align} \label{Improved hardy general d}
       \notag  I_{m,\lambda}[u] &\coloneqq \int_{\Omega} \lvert \nabla u \rvert^k X_1^{\lambda} - \left(\frac{\lambda +k-1}{k}\right)^k  \int_{\Omega} \frac{\lvert u \rvert^k}{d^k} X_1^{\lambda + k} \\
    & \hspace{60pt}- \frac{k-1}{2k} \left(\frac{\lambda+k-1}{k}\right)^{k-2} \sum_{i=2}^{m} \int_{\Omega} \frac{\lvert u \rvert^k}{ d^k} X_1^{\lambda+k} X_2^2\cdots X_i^2.
     \end{align}
     For $p=k,$ we prove the following result.
\begin{theorem} \label{p=k series expansion}
      Let $\lambda > 0$, $p=k \ge 2$ and \eqref{p=k assumption} holds. Then for all $u \in W^{1,k}_0 \left(\Omega \setminus S, X_1^{\lambda} \left(d/R_{d,\Omega}\right) \, dx \right)$ the following inequality holds:
    \begin{align}
        \notag \int_{\Omega} \lvert \nabla u \rvert^k X_1^{\lambda} & -\left(\frac{\lambda +k-1}{k}\right)^k  \int_{\Omega} \frac{\lvert u \rvert^k}{d^k} X_1^{\lambda + k}  \\
       & \hspace{20pt} \ge \frac{k-1}{2k} \left(\frac{\lambda+k-1}{k}\right)^{k-2} \sum_{i=2}^{\infty} \int_{\Omega} \frac{\lvert u \rvert^k}{ d^k} X_1^{\lambda+k} X_2^2\cdots X_i^2 \, dx.
    \end{align}
   Moreover, the constants corresponding to the Hardy-type inequality and the  $m-$ improved Hardy-type inequalities are optimal \emph{i.e.} 
    \begin{align}
        \inf_{u \in W^{1,k}_0 \left(\Omega \setminus S, X_1^{\lambda} \left(d/R_{d,\Omega}\right) \, dx \right)} \frac{\int_{\Omega} \lvert \nabla u \rvert^k X_1^{\lambda} }{\int_{\Omega} \frac{\lvert u \rvert^k}{d^k} X_1^{\lambda + k} } = \left(\frac{\lambda +k-1}{k}\right)^k,
    \end{align}
    and for each $m=2,3, \cdots$ we have,
    \begin{align}
         \inf_{u \in W^{1,k}_0 \left(\Omega \setminus S, X_1^{\lambda} \left(d/R_{d,\Omega}\right) \, dx \right)} \frac{I_{m-1,\lambda}[u] }{\int_{\Omega} \frac{\lvert u \rvert^k}{ d^k} X_1^{\lambda+k}X_2^2 \cdots X_m^2 \, dx}.
    \end{align}
\end{theorem}
Note that the present weight $X_1^{\lambda} \left(d/R_{d,\Omega}\right)$ gives  different constants from \eqref{general unweighted series expansion2}. For $p \neq k,$ we prove the following result.
\begin{theorem} \label{p not k thm}
     Let $p \neq k$, $\lambda > 0$ and \eqref{p<k assumption} holds. For all $u \in C_c^{\infty}\left(\Omega \setminus S\right)$, the following holds: For $1<p<k,$ we have,
    \begin{align} \label{intro p<k hardy}
        \int_{\Omega} \lvert \nabla u \rvert^p X_1^{\lambda} \left(\frac{d}{R_{d,\Omega}}\right) \ge \left \lvert \frac{k-p}{p}\right \rvert^p  \int_{\Omega} \frac{\lvert u \rvert^p}{ d^p} X_1^{\lambda} \left(\frac{d}{R_{d,\Omega}}\right),
    \end{align}
    and for $p>k,$ we have,
    \begin{align} \label{intro p>k hardy}
         \int_{\Omega} \lvert \nabla u \rvert^p X_1^{-\lambda} \left(\frac{d}{R_{d,\Omega}}\right) \ge \left \lvert \frac{k-p}{p}\right \rvert^p   \int_{\Omega} \frac{\lvert u \rvert^p}{ d^p} X_1^{-\lambda} \left(\frac{d}{R_{d,\Omega}}\right),
    \end{align}
  In both cases \eqref{intro p<k hardy} and \eqref{intro p>k hardy} the constant $\lvert (k-p)/p\rvert^p$ is optimal. Moreover, for $2 \le p < k$, we have the following series expansion,
  \begin{align}
     \notag  \int_{\Omega} \lvert \nabla u \rvert^p X_1^{\lambda} &- \left(\frac{k-p}{p}\right)^p  \int_{\Omega} \frac{\lvert u \rvert^p}{d^p} X_1^{\lambda}  \\
    & \hspace{40pt}\ge \frac{p-1}{2p} \left(\frac{k-p}{p}\right)^{p-2} \left(\sum_{i=1}^{\infty} \int_{\Omega} \frac{\lvert u \rvert^p}{d^p} X_1^{\lambda} X_1^2 \cdots X_i^2 \, dx\right).
  \end{align}
\end{theorem}
Furthermore, in the spirit of \cite{Adi-San, Adi-Yang}, we establish interpolation inequalities between Hardy-type inequalities and Leray-Trudinger-type inequalities \cite{adimurthi2026leraytrudingertypeexponentialintegrability}, see Section \ref{interpolation} for more details, in particular see \Cref{moser weighted1} and \Cref{moser weighted3}.
\section{Preliminaries} \label{prel}
Throughout the article, we assume that $\Omega$ is a domain in $\mathbb{R}^N$ such that $R_{d,\Omega}= \sup_{x \in \Omega} d(x) < \infty$, where $d(x) = \text{dist}(x,S)$ for a piecewise smooth closed surface $S$ of co-dimension $k$ for $1\le k \le N$. Thus, in the case of $d(x) = \lvert x \rvert,$ we have $R_{d,\Omega}= R_{\Omega} < \infty.$ We denote $\mathbb{B}_N$ as the open unit ball centered at the origin in $\mathbb{R}^N$ and $\omega_{N-1}$ as the surface area of $\partial \mathbb{B}_N.$
\par We first recall some results that will be useful in deriving the sharp forms of  Hardy inequalities in Log-weighted Sobolev spaces $W^{1,N}_0\left(\Omega\setminus \{0\},w_{\beta}\right)$, where $w_{\beta}$ is defined as in\eqref{weights w}. We will use the following Lemma from \cite[Lemma $3.1$]{adimurthi2026leraytrudingertypeexponentialintegrability}.
\begin{lemma}\label{contains lip functions}
    Let $0<\beta <\infty$. Then  $ W^{1,N}_{0}\left(\Omega, X_1^{-\beta(N-1)}\right)$ contains all locally Lipschitz functions vanishing at the boundary. 
\end{lemma}
We now introduce another weight $\ln \left(R/\lvert x \rvert\right)$ for $x \in \mathbb{B}_N$, where $R>1.$ We extend this to the whole of $\mathbb{R}^N$ by setting $1$ outside $\mathbb{B}_N,$ more precisely, we set
\begin{align*}
    E(x) =\begin{cases}
        \left(\ln \left(\frac{R}{\lvert x \rvert}\right)\right)^{N-1}, \quad \lvert x \rvert < 1\\
        1, \quad       \lvert x \rvert \ge 1.
    \end{cases}
\end{align*}
Analogous to \cite[Section 2, Proposition 1]{Sami-Jlel}, it can be shown that $D$ is  Muckenhoupt's $A_N$ weight \emph{i.e.} there exists a positive constant $c_N$ such that
\begin{align*}
    \left(\fint_{B} E(x) \, dx\right)\left(\fint_{B} E^{1/(1-N)}(x) \, dx\right)^{N-1} \le c_{N},
\end{align*}
for all balls $B \subset \mathbb{R}^N$, for more details on Muckenhoupt's weights  see \cite[Chapter $15$]{Heinonen}.  Similar to \cite[Lemma $3.1$]{adimurthi2026leraytrudingertypeexponentialintegrability}, we have the following Lemma. The proof is omitted.
\begin{lemma} \label{R vanishing at bdry}
   Let $R>1$. Then, $ W^{1,N}_{0}\left(\mathbb{B}_N, \left(\ln (R/\lvert x \rvert)\right)^{N-1}\right)$ contains all locally Lipschitz functions vanishing at the boundary. 
\end{lemma}
The following result can be seen as a particular case of \cite[Theorem $3.4$]{adimurthi2026leraytrudingertypeexponentialintegrability}.
\begin{theorem} \label{approximation}
    Let $0\le \beta \le 1$. Then, 
    \begin{align*}
        W^{1,N}_{0}\left(\Omega, X_1^{-\beta(N-1)}\right) = W^{1,N}_{0}\left(\Omega \setminus \lbrace 0 \rbrace, X_1^{-\beta(N-1)}\right).
    \end{align*}
\end{theorem}
We also need the following Hardy inequality, see \cite[Lemma 2.1]{PsaradakisJFA} for a proof.
\begin{lemma} \label{psara}
For $\gamma \in \mathbb{R}$ with $\gamma \neq 1$ and for all $u \in C^{\infty}_c\left(\Omega \setminus \{0\} \right)$ we have,
\begin{align} \label{Hardy type}
     \int_{\Omega} \lvert \nabla u \rvert^N & X_1^{-\gamma(N-1)} \, dx \ge   \left \lvert \frac{1-\gamma}{N'}\right \rvert^N \int_{\Omega} \frac{\lvert u \rvert ^N}{\lvert x \rvert^N} X_1^{N - (N-1) \gamma} \, dx.
\end{align}
\end{lemma}
Now, assume that $\Omega \cap S$ contains the origin, where $S$ is a piecewise smooth surface of co-dimension $k$ for $1 \le k \le N$. Let $d(x) = \text{dist}(x,S)$ satisfies the conditions \eqref{p<k assumption} and \eqref{p=k assumption}. Let $\eta \in C^{\infty}_c\left(B(0,\delta)\right)$ such that $\eta =1$ on $B(0,\delta/2)$.
The following technical Lemma can be found in \cite[Lemma 5.2]{Barbatis}.
\begin{lemma} \label{J_gamma}
   For $\gamma \in \mathbb{R}$ and $\varepsilon > 0$ small enough we define,
    \begin{align} \label{equation for J}
        J_{\gamma}(\varepsilon) \coloneqq \int_{\Omega} \eta^p d^{-k + \varepsilon p} X_1^{-\gamma}\left(\frac{d}{R_{d,\Omega}}\right) \, dx.
    \end{align}
    Then we have,
    \begin{enumerate}[(i)]
        \item $c \varepsilon^{-1-\gamma} \le J_{\gamma}(\varepsilon) \le C \varepsilon^{-1-\gamma}$, \quad $\gamma>-1$;
        \item $J_{\gamma}(\varepsilon) = \frac{p\varepsilon}{\gamma + 1}J_{\gamma+1}(\varepsilon) + O_{\varepsilon}(1),$ \quad $\gamma>-1$;
        \item $J_{\gamma}(\varepsilon) = O_{\varepsilon}(1),$ \quad $\gamma < -1$.
    \end{enumerate}
\end{lemma}
\section{Sharp Constants in Hardy-type inequalities} \label{sharp}
In this section, we will prove sharp Hardy-type inequalities in Log-weighted Sobolev spaces. The sharpness proofs are based on suitable choice of test functions. Similar techniques were used in \cite{ASan, Barbatis, BarbatisSeries}.
\begin{proposition} \label{beta=1}
  Let $\Omega$ be a bounded domain $\mathbb{R}^N$ containing the origin. Then, for any $u \in W^{1,N}_0\left(\Omega, X_1^{-N+1}\, dx \right)$ we have,
    \begin{align} \label{beta=1 ineq}
        \int_{\Omega} \frac{\lvert u \rvert^N}{\lvert x \rvert^N} X_2^{N} X_1 \, dx \le \left(\frac{N}{N-1} \right)^N \int_{\Omega} \lvert \nabla u \rvert^N X_1^{-N+1} \, dx.
    \end{align}
    Moreover, the constant $\left(N/(N-1)\right)^N$ is sharp.
\end{proposition}
\begin{proof}
      By \Cref{approximation}, we can restrict ourselves to $u \in C^1_c(\Omega\setminus\{0\})$. We have, 
    \begin{align*}
        -\int_{\Omega} \left \langle \nabla \left(\lvert u \rvert^N\right), \frac{x}{\lvert x \rvert^N} X_2^{N-1} \right \rangle \, dx &=  \int_{\Omega} \lvert u \rvert^N \operatorname{div}\left(\frac{x }{\lvert x \rvert^N} X_2^{N-1}\right) \, dx\\
        &=(N-1) \int_{\Omega} \lvert u \rvert^N \frac{X_2^N X_1}{\lvert x \rvert^N} \, dx.
    \end{align*}
    Thus, by using H\"older's inequality, we obtain,
    \begin{align*}
        (N-1) \int_{\Omega} \lvert u \rvert^N \frac{X_2^N X_1}{\lvert x \rvert^N} &\le N \int_{\Omega} \lvert u \rvert^{N-1} \frac{X_2^{N-1}}{\lvert x \rvert^{N-1}} X_1^{1/N'} X_1^{-1/N'} \lvert \nabla u \rvert \, dx\\
        &\le N \left(\int_{\Omega}  \frac{\lvert u \rvert^N}{\lvert x \rvert^N} X_2^{N} X_1 \, dx \right)^{1/N'} \left(\int_{\Omega} \lvert \nabla u \rvert^N X_1^{-N+1} \, dx\right)^{1/N},
    \end{align*}
    which implies \eqref{beta=1 ineq}. Finally, we prove the optimality of the constant $\left(N/(N-1)\right)^N$.
Since $\Omega$ contains the origin, so take $\delta > 0$ small enough such that $B(0,\delta) \subset \Omega.$ Since infimum over $C_c^{\infty}(\Omega)$ is smaller then infimum over $C_c^{\infty}\left(B(0,\delta)\right)$, so we can restrict ourselves to $u \in C_c^{\infty}\left(B(0,\delta)\right)$. Thus, by using scaling, we can have,
\begin{align*}
    &\inf_{u \in C_c^{\infty}\left(B(0,\delta)\right)\setminus \{0\}} \frac{\int_{B(0,\delta)} \lvert \nabla u \rvert^N \left(\ln \frac{e R_{\Omega}}{\lvert x \rvert}\right)^{N-1} \, dx}{\int_{B(0,\delta)} \frac{\lvert u \rvert ^N  \left(\ln \frac{e R_{\Omega}}{\lvert x \rvert}\right)^{-1}}{\lvert x \rvert^N \left(\ln \left(e \ln \frac{e R_{\Omega}}{\lvert x \rvert}\right)\right)^{N}} \, dx}\\
     &\qquad \qquad \qquad=\inf_{u \in W^{1,N}_0 \left(\mathbb{B}_N, \left(\ln \frac{R}{\lvert x \rvert}\right)^{N-1}\right) \setminus \{0\}} \frac{\int_{\mathbb{B}_N} \lvert \nabla u \rvert^N \left(\ln \frac{ R}{\lvert x \rvert}\right)^{N-1} \, dx}{\int_{\mathbb{B}_N} \frac{\lvert u \rvert ^N  \left(\ln \frac{ R}{\lvert x \rvert}\right)^{-1}}{\lvert x \rvert^N \left(\ln \left(e \ln \frac{ R}{\lvert x \rvert}\right)\right)^{N}} \, dx}
\end{align*}
where, $R \coloneqq eR_{\Omega}\delta^{-1} >1.$
For $0<r_1 < 1$ and $(N-1)/N < \tau < 1$, we define
\begin{align}
    u(x) = u_{r_1, \tau}(\lvert x \rvert) &\coloneqq \begin{cases}
        \left(\ln \left( e \ln \frac{R}{r_1}\right) \right)^{\tau} - \left(\ln \left( e \ln R\right)\right)^{\tau}, \quad 0 \le \lvert x \rvert \le r_1 \\
        \left(\ln \left(e \ln \frac{R}{\lvert x \rvert}\right) \right)^{\tau} - \left(\ln \left(e \ln  R\right)\right)^{\tau}, \quad r_1 \le \lvert x \rvert \le 1.
    \end{cases}
\end{align}
Thus, we have,
\begin{align*}
    \nabla u(x) &= \begin{cases}
        0, \quad 0< \lvert x \rvert < r_1\\
        -\tau \left(\ln \left(e \ln \frac{R}{\lvert x \rvert}\right) \right)^{\tau -1} \frac{1}{\ln \frac{R}{\lvert  x \rvert}}\frac{x}{\lvert x \rvert^2}, \quad r_1 < \lvert x \rvert <1.
    \end{cases}
\end{align*}
From Lemma \ref{R vanishing at bdry}, it follows that $u \in W^{1,N}_0 \left(\mathbb{B}_N, \left(\ln (R/\lvert x \rvert)\right)^{N-1}\right)$. For simplicity, we use some temporary notations: $G_{r,R} \coloneqq \ln \left(e \ln \frac{R}{r}\right)$, so in particular $G_{r_1,R} = \ln \left(e \ln \frac{R}{r_1}\right)$ and $G_{1,R}= \ln (e \ln R)$. Now, we have,
\begin{align*}
    \int_{\mathbb{B}_N}\lvert \nabla u \rvert^N \left(\ln \frac{ R}{\lvert x \rvert}\right)^{N-1} \, dx &= \omega_{N-1} \tau^N \int_{r_1}^1 G_{r,R}^{N \tau -N} \frac{1}{r \ln \frac{R}{r}} \, dr\\
    &= \frac{\omega_{N-1} \tau^N}{N \tau -N+1} \left[G_{r_1,R}^{N\tau - N + 1} - G_{1,R}^{N \tau - N +1}\right].
\end{align*}
On the other hand we have,
\begin{align*}
    &\int_{\mathbb{B}_N} \frac{\lvert u \rvert ^N \left(\ln \frac{R}{\lvert x \rvert}\right)^{-1}}{\lvert x \rvert^N \left(\ln e \ln \frac{ R}{\lvert x \rvert}\right)^{N}}\, dx\\
    &= \omega_{N-1} \left[\int^{r_1}_{0} \frac{\left[G_{r_1,R}^{\tau} - G_{1,R}^{\tau}\right]^N }{ r \ln \frac{R}{r} G_{r,R}^{N} } \, dr + \sum_{k=0}^N (-1)^k \binom{N}{k} G_{1,R}^{\tau k } \int_{r_1}^1 \frac{G_{r,R}^{\tau(N-k) -N}}{r \ln \frac{R}{r}} \, dr \right]\\
    &= \frac{\omega_{N-1}}{N-1} \sum_{k=0}^N (-1)^k \binom{N}{k}  G_{1,R}^{\tau k } G_{r_1,R}^{(N-k)\tau -N+1} \\
    &\hspace{30pt}+ \omega_{N-1} \sum_{k=0}^N \frac{(-1)^k \binom{N}{k} G_{1,R}^{\tau k }}{\tau (N-k) -N+1}\left[G_{r_1,R}^{\tau (N-k) -N+1}
    - G_{1,R}^{\tau (N-k) -N+1}\right]\\
      &= \frac{N \tau \omega_{N-1} G_{r_1,R}^{N \tau -N+1}}{(N-1)\left(N \tau - N+1\right)}+\sum_{k=1}^N c_{N,k,\tau}G_{1,R}^{\tau k }G_{r_1,R}^{(N-k)\tau - N + 1} + c_{N,\tau} G_{1,R}^{N \tau -N+1}.
\end{align*}
Therefore we have,
\begin{align*}
    \frac{\int_{\mathbb{B}_N}\lvert \nabla u \rvert^N \left(\ln \frac{ R}{\lvert x \rvert}\right)^{N-1}}{\int_{\mathbb{B}_N} \frac{\lvert u \rvert ^N \left(\ln \frac{R}{\lvert x \rvert}\right)^{-1}}{\lvert x \rvert^N \left(\ln e \ln \frac{ R}{\lvert x \rvert}\right)^{N}}} &= \frac{N-1}{N} \tau^{N-1} \left[1- \left(\frac{G_{1,R}}{G_{r_1,R}}\right)^{N \tau -N+1}\right]\\
    & \times \left[1 + \sum_{k=1}^N c_{N,k,\tau} \left(\frac{G_{1,R}}{G_{r_1,R}}\right)^{k \tau} + c_{N, \tau} \left(\frac{G_{1,R}}{G_{r_1,R}}\right)^{N \tau -N + 1} \right]^{-1},
\end{align*}
where the right hand side goes to $\left((N-1)/N\right) \tau^{N-1}$ as $r_1 \to 0.$ Now, we can take $\tau \to (N-1)/N$ to conclude the result.
\end{proof}
Let us now deal with the case $ \beta \neq1$ with $\beta > 0.$
\begin{proposition} \label{sharp hardy type1}
 Let $\Omega$ be bounded domain $\mathbb{R}^N$ containing the origin. Consider $\beta > 0$ and $\beta \neq 1$, then for all $u \in W^{1,N}_{0}\left(\Omega \setminus \{0\}, X_1^{-\beta(N-1)}\right)$, we have the following inequality,
    \begin{align} \label{hardy type X_1}
        \int_{\Omega} \frac{\lvert u \rvert ^N}{\lvert x \rvert^N} X_1^{N - (N-1) \beta} \, dx \le \frac{1}{\lvert H \rvert^N}  \int_{\Omega} \lvert \nabla u \rvert^N X_1^{-\beta(N-1)} \, dx,
    \end{align}
    where $H = (1-\beta)/N'$. Moreover, then the constant $1/\lvert H \rvert^N$  is sharp.
\end{proposition}
\begin{proof}
       We show the optimality of the constant as the proof of the inequality \eqref{hardy type X_1} follows from Lemma \ref{psara}. Let $\delta>0$ be such that $B(0,\delta) \subset \Omega.$ Now, consider $\eta \in C_c^{\infty}\left(B(0,\delta)\right)$ such that $0\le \eta \le 1$ and $\eta = 1$ on $B\left(0,\delta/2\right)$.
For any small enough $\varepsilon>0$ and $\tau > 0$ we set $v_{\varepsilon} (x)\coloneqq \lvert x \rvert^{\varepsilon} X_1^{-H-\tau}\left(\lvert x \rvert/R_{\Omega}\right)$ and $V_{\varepsilon} \coloneqq \eta v_{\varepsilon}$. $V_{\varepsilon}$ is locally lipschitz that vanishes at the boundary of $\Omega \setminus \{0\}.$ Thus, from Lemma \ref{contains lip functions}, it follows that $V_{\varepsilon} \in W^{1,N}_{0}\left(\Omega \setminus \lbrace 0 \rbrace, X_1^{-\beta(N-1)}\right).$ Note that, $\nabla V_{\varepsilon} = \eta \nabla v_{\varepsilon} + v_{\varepsilon} \nabla \eta$ and using an elementary inequality:
\begin{align} \label{elt ineq1}
    \lvert a + b \rvert^p \le \lvert a \rvert^p + c_p \left(\lvert b \rvert^p + \lvert b \rvert \lvert a \rvert^{p-1}\right), \quad a,b \in \mathbb{R}^N, \quad p>1;
\end{align} we deduce that,
\begin{align*}
    \int_{\Omega}& \lvert \nabla V_{\varepsilon}\rvert^N X_1^{-\beta(N-1)} \, dx\\ &\le \int_{B(0,\delta)} \eta^N \lvert \nabla v_{\varepsilon} \rvert^N X_1^{-\beta(N-1)} \, dx + c_N \int_{B(0,\delta)} v_{\varepsilon}^N \lvert \nabla \eta \rvert^N X_1^{-\beta(N-1)} \, dx \\ &\hspace{20pt}+ c_N \int_{B(0,\delta)} v_{\varepsilon} \lvert \nabla v_{\varepsilon} \rvert^{N-1} \eta^{N-1} \lvert \nabla \eta \rvert X_1^{-\beta(N-1)} \, dx=  I_1 + I_2 + I_3.
\end{align*}
 We now show that $I_2, I_3 = O_{\varepsilon}(1)$ as $\varepsilon \to 0.$ Observe that,
\begin{align} \label{v_epsilon}
\nabla v_{\varepsilon} = \lvert x \rvert^{\varepsilon -1} X_1^{-H-\tau} \left(\varepsilon +(-H - \tau) X_1\right) \frac{x}{\lvert x \rvert}.
\end{align}
Note that $\varepsilon +(-H- \tau) X_1$ is bounded on $B(0,\delta).$
Thus, since $\eta$ is smooth, we have,
\begin{align*}
    I_3 &\le \Tilde{c} \int_{B(0,\delta)} \lvert x \rvert^{N \varepsilon -N +1} X_1^{-NH-N \tau - \beta(N-1)} \, dx = O_{\varepsilon}(1) \quad \text{as} \quad \varepsilon \to 0.
\end{align*}
Similarly, we can show that $I_2 = O_{\varepsilon}(1)$ as $\varepsilon \to 0.$ From \Cref{v_epsilon} we have,
\begin{align*}
    \lvert \nabla v_{\varepsilon} \rvert^N \le \lvert x \rvert^{N \varepsilon -N} X_1^{(\beta-1)(N-1)-N \tau}  \sum_{k=0}^N \binom{N}{k} \varepsilon^{N-k} \lvert H +\tau \rvert^k X_1^k.
\end{align*}
Thus, using the above relation, we ended up with,
\begin{align*}
    I_1 &\le \sum_{k=0}^N \binom{N}{k} \varepsilon^{N-k} \lvert H +\tau \rvert^k\int_{B(0,\delta)} \eta^N \lvert x \rvert^{N \varepsilon -N}  X_1^{-N \tau -N+1 +k}\\
    &= \sum_{k=0}^N \binom{N}{k} \varepsilon^{N-k} \lvert H +\tau \rvert^k J_{N \tau +N-1 -k}(\varepsilon),
\end{align*}
where $J$ is defined as in \eqref{equation for J} for the case $d(x) = \lvert x \rvert.$ Since, by our choice, $\tau >0$ so we have $(N \tau + N-1) -N >-1$. Thus, $(ii)$ of Lemma \ref{J_gamma}  and induction yields,
\begin{align*}
     \varepsilon^{N-k}  J_{N \tau +N-1 -k}(\varepsilon) = \left(\tau + 1 - \frac{1}{N} - \frac{k}{N}\right) &\left(\tau + 1 - \frac{1}{N} - \frac{k+1}{N}\right) \cdots\\
     & \left(\tau + 1 - \frac{1}{N} - \frac{N-1}{N}\right) J_{N \tau - 1} + O_{\varepsilon}(1),
\end{align*}
for $k=0, \cdots N-1.$  On the other hand, we have,
\begin{align*}
    \int_{\Omega} \frac{\lvert V_{\varepsilon} \rvert^N}{\lvert x \rvert^N}X_1^{N-(N-1)\beta}\, dx &= \int_{\Omega} \eta^N \lvert x \rvert^{N \varepsilon - N} X_1^{-NH-N\tau+N-(N-1)\beta}\, dx= J_{N \tau -1}(\varepsilon).
\end{align*}
Therefore, we have
\begin{align*}
    &\frac{\int_{\Omega} \lvert \nabla V_{\varepsilon}\rvert^N X_1^{-\beta(N-1)} \, dx}{ \int_{\Omega} \frac{\lvert V_{\varepsilon} \rvert^N}{\lvert x \rvert^N} X_1^{N-(N-1)\beta}\, dx}\\
    &\hspace{40pt}\le \lvert H + \tau \rvert^N + O_{\varepsilon}(1) J^{-1}_{N \tau -1 }(\varepsilon)\\
    & \hspace{40pt} \quad +\sum_{k=0}^{N-1} \lvert H + \tau \rvert^N \binom{N}{k} \left(\tau + 1 - \frac{1}{N} - \frac{k}{N}\right) \left(\tau + 1 - \frac{1}{N} - \frac{k+1}{N}\right) \cdots\tau.
\end{align*}
Now, by $(i)$ of \Cref{J_gamma} we have $J_{N \tau  -1}(\varepsilon) \to \infty$ as $\varepsilon \to 0$. Thus, first letting $\varepsilon \to 0$ and then $\tau \to 0$ in the above expression yields the sharpness of the constant.
\end{proof}
Let us now consider the case of distance from a piecewise smooth closed surface $S$. Without loss of generality assume that $0 \in S \cap \Omega$. We will write $X_j$ to mean $X_j\left(d/R_{d,\Omega}\right)$ for $j=1,2,\cdots$. We now prove the following Lemma.
\begin{lemma} \label{loc lip cpt support}
    Suppose $u$ is a locally lipschitz function defned on $\Omega$ that has compact support then $u \in W_0^{1,p}\left(\Omega \setminus S ,X_1^{\gamma}(\left(d/R_{d,\Omega}\right)\right)$ for any $\gamma \in \mathbb{R}$ .
\end{lemma}
\begin{proof}
    We have $suppu =K$ is compact. We can assume that $u $ is bounded. Let $\rho_{\varepsilon}$ be the standard family of mollifiers and take $\varepsilon$ small enough such that $u_{\varepsilon} \coloneqq \rho_{\varepsilon} * u \in C_c^{\infty}(\Omega).$ Now, $u_{\varepsilon}  \to u$ and $\nabla u_{\varepsilon}  \to \nabla u$ a.e. and $\| u_{\varepsilon}  \| \le \|u \|_{L^{\infty}}$,  $\| \nabla u_{\varepsilon}  \| \le \| \nabla u \|_{L^{\infty}}$. Let us show that $X_1^{\gamma}(\left(d/R_{d,\Omega}\right)$ is integrable on $K$ for any $\gamma \in \mathbb{R}.$ Without loss of generality we can assume that $0 \in \Omega \setminus S,$ so take $\delta > 0$ such that $B(0,\delta) \subset \Omega$. Note that, it's enough to check at $B(0,\delta).$ By the coarea formula and the fact $\int_{B(0,\delta) \cap \{d=r\}} d \mathcal{H}^{N-1}(x) \le c r^{k-1}$ we have,
    \begin{align*}
        \int_{B(0,\delta)} X_1^{\gamma}\left(\frac{d}{R_{d,\Omega}}\right) \, dx &= \int_{0}^{\delta} \int_{d=r}  X_1^{\gamma}\left(\frac{r}{R_{d,\Omega}}\right)  d \mathcal{H}^{N-1} \, dr\\
        &\le c \int_{0}^{\delta} X_1^{\gamma}\left(\frac{r}{R_{d,\Omega}}\right) r^{k-1} \, dr < \infty.
    \end{align*}
    Thus, dominated convergence Theorem yields that $u_{\varepsilon} \to u$ in $W^{1,p}_0(\Omega \setminus S,X_1^{\gamma}).$
\end{proof}
\begin{proposition} \label{p<k ineq hardy prop}
    Let $p \neq k$, $\lambda > 0$ and \eqref{p<k assumption} holds. Denote, $M \coloneqq (k-p)/p.$ Then  for all $u \in C_c^{\infty}\left(\Omega \setminus S\right)$, the following holds: For $1<p<k,$ we have,
    \begin{align} \label{p<k ineq hardy}
         \int_{\Omega} \lvert \nabla u \rvert^p X_1^{\lambda} \left(\frac{d}{R_{d,\Omega}}\right) \ge \lvert M \rvert^p  \int_{\Omega} \frac{\lvert u \rvert^p}{ d^p} X_1^{\lambda} \left(\frac{d}{R_{d,\Omega}}\right),
    \end{align}
    and for $p>k,$ we have,
    \begin{align} \label{p>k Hardy}
        \int_{\Omega} \lvert \nabla u \rvert^p X_1^{-\lambda} \left(\frac{d}{R_{d,\Omega}}\right)  \ge \lvert M \rvert^p  \int_{\Omega} \frac{\lvert u \rvert^p}{ d^p} X_1^{-\lambda}\left(\frac{d}{R_{d,\Omega}}\right),
    \end{align}
  In both the cases \eqref{p<k ineq hardy} and \eqref{p>k Hardy}, the constant $\lvert M \rvert^p$ is optimal.
\end{proposition}
\begin{proof}
Let us first prove the inequality \eqref{p<k ineq hardy}. Let $V$ be a $C^1$ vector field defined on $\Omega \setminus S$. Now, for $u \in C_c^{\infty}\left(\Omega \setminus S\right)$, we have,
    \begin{align} \label{div eqn1}
     \notag  \int_{\Omega} \operatorname{div} \left(X_1^{\lambda} V\right) \lvert u \rvert^p \, dx &= -p \int_{\Omega}\langle X_1^{\lambda} V, \nabla u \rangle \lvert u \rvert^{p-2} u \, dx \\
    \notag &\le p \int_{\Omega} X_1^{\lambda/p} \lvert V \rvert \lvert \nabla u \rvert \lvert u \rvert^{p-1} X_1^{\lambda/p'} \, dx\\
       &\le \int_{\Omega} \lvert \nabla u \rvert^p X_1^{\lambda} \, dx + (p-1)\int_{\Omega} \lvert V \rvert^{p'} \lvert u \rvert^p X_1^{\lambda} \, dx,
    \end{align}
    where in the last step we have used Young's inequality: $ab \le a^p/p + b^{p'}/p'$ for any $a,b \ge 0$ with $p,p'$ are H\"older conjugates.
This implies that,
\begin{align} \label{equation for p<k}
    \int_{\Omega} \lvert \nabla u \rvert^p X_1^{\lambda} \, dx \ge \int_{\Omega} \left[\operatorname{div} \left(X_1^{\lambda} V\right) - (p-1) \lvert V \rvert^{p'} X_1^{\lambda}\right] \lvert u \rvert^p \, dx.
\end{align}
Thus, it's enough to show that, $\operatorname{div} \left(X_1^{\lambda} V\right) - (p-1) \lvert V \rvert^{p'} X_1^{\lambda} \ge M^pd^{-p}X_1^{\lambda}$. Consider the vector field $V(x) = M^{p-1}{d^{-p+1}(x)} \nabla d(x)$. Thus, by using \eqref{p<k assumption} we have,
\begin{align*}
    \operatorname{div} V &= M^{p-1} \left[\frac{\Delta d}{d^{p-1}(x)} + (1-p) \frac{\lvert \nabla d \rvert^2}{d^p}\right]\\
    &= M^{p-1} \frac{d \Delta d- (p-1)}{d^p}\stackrel{\eqref{p<k assumption}}{\ge} M^{p-1} \frac{k-p}{d^p} =p \frac{M^p}{d^p}.
\end{align*}
This implies that,
\begin{align*}
    \operatorname{div}(X_1^{\lambda} V) &= X_1^{\lambda} \operatorname{div}V + \left \langle V, \lambda d^{-1} X_1^{\lambda + 1} \nabla d  \right \rangle \ge p M^p d^{-p} X_1^{\lambda}.
\end{align*}
On the other hand, we have, $\lvert V \rvert^{p'} X_1^{\lambda} = M^p d^{-p}X_1^{\lambda}$. Therefore, we have $\operatorname{div} \left(X_1^{\lambda} V\right) - (p-1) \lvert V \rvert^{p'} X_1^{\lambda} \ge M^pd^{-p}X_1^{\lambda}$ and \eqref{p<k ineq hardy} follows. To prove \eqref{p>k Hardy}, we can proceed similarly as above by choosing the vector field, $V(x)= M\lvert M \rvert^{p-2} d^{-p+1} \nabla d(x).$
\par Now, we prove the sharpness of the constant. Again we consider $\eta \in C_c^{\infty}\left(B(0,\delta)\right)$ such that $0\le \eta \le 1$ and $\eta = 1$ on $B\left(0,\delta/2\right)$.
For any $\varepsilon>0$, we set $v_{\varepsilon} \coloneqq d^{-M+\varepsilon} X_1^{ - \lambda/p}\left(d/R_{d,\Omega}\right)$ and $V_{\varepsilon} \coloneqq \eta v_{\varepsilon}$. Take $\phi \in C^{\infty}([0,\infty))$ such that $\phi(t) = 0$ for $t \le 1$ and $\phi(t) = 1,$ for $t \ge 2.$ For $\tau>0,$ set $\phi_{\tau}(x) = \phi ( d(x)/\tau).$ Thus,
\begin{align} \label{cut off functions}
    \phi_{\tau}(x) = \begin{cases}
        0, \quad d(x) \le \tau \\
        1, \quad d(x) \ge 2 \tau.
    \end{cases}
\end{align}
Observe that, $\lvert \nabla \phi_{\tau} \rvert \le c/\tau.$ Set $V_{\tau} = \phi_{\tau} V_{\varepsilon},$ note that $V_{\tau}$ is locally lipschitz and has compact support, so by Lemma \ref{loc lip cpt support},  $V_{\tau} \in W^{1,p}_0(\Omega \setminus S, X_1^{\lambda} ).$ Note that we have,
\begin{align*}
    \int_{\Omega} \lvert \nabla V_{\tau} - \nabla V_{\varepsilon} \rvert^p X_1^{\lambda} \, dx &= C_p\int_{\Omega} \lvert \phi_\tau -1 \rvert^p \lvert \nabla V_{\varepsilon} \rvert^p X_1^{\lambda} \, dx +C_p  \int_{\Omega} \lvert \nabla \phi_\tau  \rvert^p \lvert V_{\varepsilon} \rvert^p X_1^{\lambda} \, dx\\
    &\coloneqq  K_1 + K_2.
\end{align*}
Now, by dominated convergence Theorem, $K_1 \to 0,$ as $\tau \to 0.$ Now, 
\begin{align*}
    K_2 & = C_p  \int_{\Omega} \lvert \nabla \phi_\tau  \rvert^p \lvert V_{\varepsilon} \rvert^p X_1^{\lambda} \, dx \le C \tau^{-p} \int_{B(0,\delta) \cap \{ \tau <  d < 2 \tau \} } d^{-k+p+\varepsilon p} \, dx = \Tilde{C} \tau^{\varepsilon p},
    \end{align*}
    which goes to $0,$ as $\tau \to 0,$ where in the last step we have used the coarea formula together the fact that $\int_{B(0,\delta) \cap \{d=r\}} d \mathcal{H}^{N-1}(x) \le c r^{k-1}.$ Thus, $V_{\varepsilon} \in W^{1,p}_0\left(\Omega \setminus S, X_1^{\lambda}\right)$.
\par Note that, $\nabla V_{\varepsilon} = \eta \nabla v_{\varepsilon} + v_{\varepsilon} \nabla \eta$. Thus, an elementary inequality \eqref{elt ineq1} yields,
\begin{align*}
    \int_{\Omega} \lvert \nabla V_{\varepsilon}\rvert^p X_1^{\lambda} \, dx \le \int_{B(0,\delta)} &\eta^p \lvert \nabla v_{\varepsilon} \rvert^p X_1^{\lambda} \, dx+ c_p \int_{B(0,\delta)} v_{\varepsilon}^p \lvert \nabla \eta \rvert^p X_1^{\lambda} \, dx\\
    &+ c_p \int_{B(0,\delta)} v_{\varepsilon} \lvert \nabla v_{\varepsilon} \rvert^{p-1} \eta^{p-1} \lvert \nabla \eta \rvert X_1^{\lambda} \, dx \coloneqq I_1 + I_2 + I_3.
\end{align*}
We show that $I_2, I_3 = O_{\varepsilon}(1)$ as $\varepsilon \to 0.$
Observe that,
\begin{align} \label{v_epsilon second}
\lvert \nabla v_{\varepsilon} \rvert \le \lvert M-\varepsilon \rvert d^{-M+\varepsilon -1} X_1^{- \lambda/p} + C_{p,\lambda} d^{-M+\varepsilon -1} X_1^{- \lambda/p+1} .
\end{align}
Note that we have,
\begin{align*}
    I_3 &= c_p \int_{B(0,\delta)} v_{\varepsilon} \lvert \nabla v_{\varepsilon} \rvert^{p-1} \eta^{p-1} \lvert \nabla \eta \rvert X_1^{\lambda} \, dx \le C \int_{B(0,\delta)} d^{1-k+ p \varepsilon} \,dx = O_{\varepsilon}(1).
\end{align*}
Similarly we can show that $I_3 = O_{\varepsilon}(1)$. Again we use \eqref{elt ineq1} to \eqref{v_epsilon second} and deduce,
\begin{align*}
    I_1 &\le (M-\varepsilon)^p \int_{B(0,\delta)} \eta^p d^{-k+ \varepsilon p} \, dx + c_{p,\lambda} \int_{B(0,\delta)}  \eta^p d^{-k+ \varepsilon p} X_1^p \, dx.\\
    &\le (M-\varepsilon)^p \int_{\Omega} \frac{\lvert V_{\varepsilon} \rvert^p}{d^p} X_1^{\lambda} \, dx + O_{\varepsilon}(1)=(M-\varepsilon)^p J_{0}(\varepsilon) + O_{\varepsilon}(1),
\end{align*}
where $J$ is defined as in \eqref{equation for J}. Therefore, we have,
\begin{align*}
    \frac{\int_{\Omega} \lvert \nabla V_{\varepsilon}\rvert^p X_1^{\lambda} \, dx}{\int_{\Omega} \frac{\lvert V_{\varepsilon} \rvert^p}{d^p} X_1^{\lambda} \, dx} \le (M-\varepsilon)^p  + J^{-1}_{0}(\varepsilon)O_{\varepsilon}(1).
\end{align*}
Letting $\varepsilon \to 0,$ by $(i)$ of \Cref{J_gamma} we have $J_{0}(\varepsilon) \to \infty$ and the sharpness follows.
\par To prove the sharpness in \eqref{p>k Hardy}, we choose $v_{\varepsilon}= d^{-M+ \varepsilon} X_1^{\lambda/p}(d/R_{d,\Omega})$ and proceed similarly. Thus, we omit the details.
\end{proof}
\begin{proposition} \label{p=k hardy ineq prop}
    Let $\lambda > 0$, $p=k \ge 2$ and \eqref{p=k assumption} holds. Then for all $u \in W^{1,k}_0(\Omega \setminus S, X_1^{\lambda} \left(d/R_{d,\Omega}\right) \, dx )$, we have the following inequality,
    \begin{align} \label{p=k hardy ineq}
        \int_{\Omega} \lvert \nabla u \rvert^k X_1^{\lambda} \left(\frac{d}{R_{d,\Omega}}\right) \, dx \ge \left(\frac{\lambda + k-1}{k}\right)^k  \int_{\Omega} \frac{\lvert u \rvert^k}{d^k} X_1^{\lambda + k} \left(\frac{d}{R_{d,\Omega}}\right) \, dx.
    \end{align}
    Moreover, the constant $(\lambda + k -1)^k/ k^k$ is optimal.
\end{proposition}
\begin{proof}
  First, we prove the inequality \eqref{p=k hardy ineq}. For this purpose, take $V$ to be $C^1$ vector field on $\Omega \setminus S.$ For all $u \in C^{\infty}_c\left(\Omega \setminus S\right)$ and any $\gamma \in \mathbb{R}$ we have,
\begin{align} \label{divergence eqn k gen alpha}
   \notag \int_{\Omega} \operatorname{div} V \lvert u \rvert^k \, dx &= - k \int_{\Omega} \langle V, \nabla u \rangle \lvert u \rvert^{k-2} u \, dx\\
  \notag &\le k \int_{\Omega} \lvert \nabla u \rvert X_1^{\gamma/k} \lvert V \rvert \lvert u \rvert^{k-1} X_1^{-\gamma/k} \,dx\\
    &\le \int_{\Omega} \lvert \nabla u \rvert^k X_1^{\gamma} \, dx + (k-1) \int_{\Omega} \lvert V \rvert^{k'} \lvert u \rvert^k X_1^{-\gamma/(k-1)} \, dx,
\end{align}
where the last step follows from Young's inequality.
Thus, in particular we have,
\begin{align} \label{p=k eqn}
    \int_{\Omega} \lvert \nabla u \rvert^k X_1^{\lambda} \, dx \ge \int_{\Omega} \left(\operatorname{div} V-(k-1) \lvert V \rvert^{k'} X_1^{-\lambda/(k-1)}\right) \lvert u \rvert^{k} \, dx.
\end{align}
Now, consider the vector field,
\begin{align*}
    V(x) = L^{k-1} \nabla d(x)d^{-k+1}(x) X_1^{\lambda + k-1}, \mbox{ for } L = (\lambda+k-1)/k.
\end{align*}
Thus, we have,
\begin{align*}
    \operatorname{div} V(x) &= L^{k-1} d^{-k} X_1^{\lambda +k-1} \left[(\lambda + k-1) X_1-k+1+ d\Delta d\right] \\
    & \stackrel{\eqref{p=k assumption}}{\ge} L^{k-1} (\lambda + k-1) d^{-k} X_1^{\lambda +k}= k L^k d^{-k} X_1^{\lambda +k}.
\end{align*}
Therefore by using the above computation, from \eqref{p=k eqn} we have,
\begin{align*}
    \int_{\Omega} \lvert \nabla u \rvert^k X_1^{\lambda} \, dx &\ge L^k\int_{\Omega} \frac{1}{d^k}\left(k X_1^{\lambda +k} - (k-1) X_1^{\frac{k(\lambda+k-1)}{k-1}}X_1^{\frac{-\lambda}{k-1}}\right) \lvert u \rvert^k \, dx\\
    &=L^k \int_{\Omega} \frac{\lvert u \rvert^k}{d^k} X_1^{\lambda + k} \, dx.
\end{align*}
\par We still need to prove the optimality of the constant. Without loss of generality, assume that $0 \in \Omega \cap S$, so take a $\delta > 0$ such that $B(0,\delta) \subset \Omega.$ Now, consider $\eta \in C_c^{\infty}\left(B(0,\delta)\right)$ such that $0\le \eta \le 1$ and $\eta = 1$ on $B\left(0,\delta/2\right)$. 
For any small enough $\varepsilon>0$ and $\tau > 0$ we set $v_{\varepsilon} \coloneqq d^{\varepsilon} X_1^{-L-\tau}$ and $V_{\varepsilon} \coloneqq \eta v_{\varepsilon}$. We set $V_{\tau}=\phi_{\tau} V_{\varepsilon},$ where $\phi_{\tau}$ is defined as in \eqref{cut off functions}. Analogous to the proof of Proposition \ref{p<k ineq hardy prop}, we can show that $V_{\varepsilon} \in W^{1,k}_0(\Omega \setminus S, X_1^{\lambda}).$ Note that, $\nabla V_{\varepsilon} = \eta \nabla v_{\varepsilon} + v_{\varepsilon} \nabla \eta$ and using an elementary inequality \eqref{elt ineq1}, we deduce,
\begin{align*}
      \int_{\Omega} \lvert \nabla V_{\varepsilon}\rvert^k X_1^{\lambda} \, dx \le \int_{B(0,\delta)} &\eta^k \lvert \nabla v_{\varepsilon} \rvert^k X_1^{\lambda} \, dx + c_k \int_{B(0,\delta)} v_{\varepsilon}^k \lvert \nabla \eta \rvert^k X_1^{\lambda} \, dx\\
      &+ c_k \int_{B(0,\delta)} v_{\varepsilon} \lvert \nabla v_{\varepsilon} \rvert^{k-1} \eta^{k-1} \lvert \nabla \eta \rvert X_1^{\lambda} \, dx \coloneqq I_1 + I_2 + I_3.
\end{align*}
We show that $I_2, I_3 = O_{\varepsilon}(1)$ as $\varepsilon \to 0.$ Observe that,
\begin{align} \label{v_epsilond}
\nabla v_{\varepsilon} =  d^{\varepsilon -1} X_1^{-L-\tau} \left(\varepsilon +(-L - \tau) X_1\right) \nabla d.
\end{align}
Note that $\varepsilon +(-L- \tau) X_1$ is bounded on $B(0,\delta).$
Thus, we can deduce that,
\begin{align*}
    I_3 \le c \int_{B(0,\delta)} d^{1-k+ \varepsilon k} X_1^{1-k-\tau k}(d/R_{\Omega}) \, dx = O_{\varepsilon}(1) \quad \text{as} \quad \varepsilon \to 0.
\end{align*}
Similarly, we can show that $I_2 = O_{\varepsilon}(1)$ as $\varepsilon \to 0.$
From \Cref{v_epsilond} we have,
\begin{align*}
    \lvert \nabla v_{\varepsilon} \rvert^k
    &\le d^{k \varepsilon -k} X_1^{1-k-\lambda-k\tau}  \sum_{i=0}^k \binom{k}{i} \varepsilon^{k-i} \lvert L +\tau \rvert^i X_1^i.
\end{align*}
Thus, we have,
\begin{align*}
    I_1 &\le \sum_{i=0}^k \binom{k}{i} \varepsilon^{k-i} \lvert L +\tau \rvert^i J_{-1+k+k \tau -i}; \hspace{4pt} J_{-1+k+k \tau -i}= \int_{B(0,\delta)} \eta^kd^{k\varepsilon-k}X_1^{1-k-k\tau+i}.
\end{align*}
Since, by our choice, $\tau >0$ so we have $(k \tau + k-1) -k >-1$. Thus, by $(ii)$ of Lemma \ref{J_gamma}  and using induction we can deduce,
\begin{align*}
     \varepsilon^{k-i}  J_{k \tau +k-1 -i}(\varepsilon) = \left(\tau + 1 - \frac{1}{k} - \frac{i}{k}\right) &\left(\tau + 1 - \frac{1}{k} - \frac{i+1}{k}\right) \cdots\\
     &\left(\tau + 1 - \frac{1}{k} - \frac{k-1}{k}\right) J_{k \tau - 1} + O_{\varepsilon}(1),
\end{align*}
for $i=0, \cdots k-1.$  On the other hand, we have,
\begin{align*}
    \int_{\Omega} \frac{\lvert V_{\varepsilon} \rvert^k}{d^k}X_1^{\lambda+k}\, dx &= \int_{\Omega} \eta^k d^{k \varepsilon - k} X_1^{1-k \tau}\, dx= J_{k \tau -1}(\varepsilon),
\end{align*}
where $J$ is defined as in \eqref{equation for J}. Therefore, we have
\begin{align*}
    \frac{\int_{\Omega} \lvert \nabla V_{\varepsilon}\rvert^k X_1^{\lambda} \, dx}{ \int_{\Omega} \frac{\lvert V_{\varepsilon} \rvert^k}{d^k} X_1^{\lambda + k}\, dx}\\
    &\le \lvert M + \tau \rvert^k + O_{\varepsilon}(1) J^{-1}_{k \tau -1 }(\varepsilon)\\
    & \qquad \qquad +\sum_{i=0}^{k-1}\binom{k}{i} \left(\tau + 1 - \frac{1}{k} - \frac{i}{k}\right) \left(\tau + 1 - \frac{1}{k} - \frac{i+1}{k}\right) \cdots\tau.
\end{align*}
Now, by $(i)$ of Lemma \ref{J_gamma} we have $J_{k \tau  -1}(\varepsilon) \to \infty$ as $\varepsilon \to 0$. Thus, first letting $\varepsilon \to 0$ and then $\tau \to 0$ in the above expression yields the sharpness of the constant.
\end{proof}
\section{Series Expansion} \label{expansion}
In this section, we will prove that we can improve the Hardy-type inequalities by adding suitable terms and get the infinite series expansion of remainder terms. We will also give the proof of \Cref{p not k thm}. The proofs are based on the suitable choice of the vector field. This approach was first developed in \cite{Barbatis} and was further studied in \cite{BarbatisSeries}.
\begin{proposition} \label{series expansion without beta}
     For all $u \in W^{1,N}_{0}\left(\Omega, X_1^{-N+1}\right)$, we have the following inequality,
    \begin{align}
        \notag \int_{\Omega} \lvert \nabla u \rvert^N X_1^{-N+1} \, dx &- \left(\frac{N-1}{N}\right)^N \int_{\Omega} \frac{\lvert u \rvert ^N}{\lvert x \rvert^N} X_2^N X_1 \, dx\\
 & \quad \ge \frac{1}{2}\left(\frac{N-1}{N}\right)^{N-1} \sum_{j=3}^{\infty} \int_{\Omega} \frac{\lvert u \rvert ^N}{\lvert x \rvert^N} X_2^N X_1 X_3^2 \cdots X_j^2 \, dx.
    \end{align}
\end{proposition}
\begin{proof} 
From \Cref{approximation}, we can assume that $u \in C^{\infty}_c\left(\Omega \setminus \{0\}\right)$. Let $V$ be a $C^1$ vector field defined on $\Omega \setminus \{0\}$. Putting $k=N$ and $\gamma = -N+1$ in \eqref{divergence eqn k gen alpha} yields,
    \begin{align} \label{expansion eq}
     \int_{\Omega} \lvert \nabla u \rvert^N  X_1^{-N+1} \, dx \ge  \int_{\Omega} \left(\operatorname{div} V -(N-1) \lvert V \rvert^{N'} X_1\right) \lvert u \rvert^N \, dx.
\end{align}
Consider the following vector field.
    \begin{align*}
        V(x) = \left(\frac{N-1}{N}\right)^{N-1} \left(X_2^{N-1} + \sum_{j=3}^m X_2^{N-1} X_3 \cdots X_j\right) \frac{x}{\lvert x \rvert^N}
    \end{align*}
Thus, we have,
\begin{align} \label{div eqn X}
   \notag  \operatorname{div} V &= \frac{(N-1)^N}{N^{N-1}} \frac{X_2^NX_1}{\lvert x \rvert^N} \Bigg[1+ \sum_{j=3}^m X_3 \cdots X_j\\
   &\hspace{100pt}+\frac{1}{N-1}\sum_{j=3}^m \sum_{i=3}^j X_3^2 \cdots X_i^2X_{i+1} \cdots X_j \Bigg].
\end{align}
On the other hand,
\begin{align} \label{V eqn X}
 \lvert V \rvert^{N'} X_1 = \left(\frac{N-1}{N}\right)^N  \frac{X_2^NX_1}{\lvert x \rvert^N} \left[1+ \sum_{j=3}^m X_3\cdots X_j\right]^{N'}.
 \end{align}
 Now, by Taylor series expansion we can estimate the last term as,
\begin{align*}
    \left[1+ \sum_{j=3}^m X_3 \cdots X_j\right]^{N'} &\le 1 + N'\sum_{j=3}^m X_3 \cdots X_j +\frac{N}{2(N-1)^2} \left(\sum_{j=3}^m X_3 \cdots X_j\right)^2.
\end{align*} 
Thus, from \eqref{div eqn X}, \eqref{V eqn X} and using the above expansion, we can deduce,
\begin{align*}
    &\operatorname{div} V -(N-1) \lvert V \rvert^{N'} X_1 \\ 
    &\qquad \ge \left( \frac{N-1}{N}\right)^{N-1} \frac{X_2^N X_1}{\lvert x \rvert^N} \\ 
    &\hspace{50pt}\times\left[\frac{N-1}{N}+\sum_{j=3}^m \sum_{i=3}^j X_3^2 \cdots X_i^2X_{i+1} \cdots X_j -\frac{1}{2} \left(\sum_{j=3}^m X_3 \cdots X_j\right)^2\right].
\end{align*}
Expanding the square in the last term of the above equation we ended up with,
\begin{align*}
    \operatorname{div} V -&(N-1) \lvert V \rvert^{N'} X_1 \ge \left( \frac{N-1}{N}\right)^{N-1} \frac{X_2^N X_1}{\lvert x \rvert^N}\left[\frac{N-1}{N} +\frac{1}{2} \sum_{j=3}^m X_3^2 \cdots X_j^2\right].
\end{align*}
The above equation together with \eqref{expansion eq} implies the proof.
\end{proof}
\begin{proposition} \label{only series beta}
Let $0 < \beta < 1.$ Then for  all $u \in W^{1,N}_{0}\left(\Omega , w_{\beta}\right)$ , we have the following infinite improvement of remainder terms:
    \begin{align}
      \notag   \int_{\Omega} \lvert \nabla u \rvert^N & X_1^{-\beta(N-1)} \, dx - \left(\frac{1-\beta}{N'}\right)^N \int_{\Omega} \frac{\lvert u \rvert ^N}{\lvert x \rvert^N} X_1^{ N - \beta(N-1)} \, dx\\
        & \quad \ge \frac{1}{2(1-\beta)}\left(\frac{1-\beta}{N'}\right)^{N-1} \sum_{j=2}^{\infty} \int_{\Omega} \frac{\lvert u \rvert ^N}{\lvert x \rvert^N} X_1^{N- \beta(N-1)} X_2^2 \cdots X_j^2 \, dx
        \label{beta series prop}.
    \end{align}
\end{proposition}
\begin{proof}
By \Cref{approximation}, we can assume that $u \in C_c^{\infty}\left(\Omega \setminus \{0\}\right)$. Take any $C^1$ vector field $V$ defined on $\Omega \setminus \{0\}$. Now, putting $k=N$ and $\gamma = -\beta(N-1)$ in \eqref{divergence eqn k gen alpha} yields,
\begin{equation}
     \label{div conclusion conclusion beta}
    \int_{\Omega} \lvert \nabla u \rvert^N X_1^{-\beta(N-1)} \, dx \ge \int_{\Omega} \left(\operatorname{div}V - (N-1)\lvert V \rvert^{N'}X_1^{\beta}\right) \lvert u \rvert^N \, dx.
\end{equation}
Consider the vector field defined by,
\begin{align}
    V = \left(\frac{1-\beta}{N'}\right)^{N-1} X_1^{(N-1)(1- \beta)}  \left( 1 + \frac{1}{(1-\beta)}\sum_{j=2}^m  X_2 \cdots X_j\right)  \frac{x}{\lvert x \rvert^N} .
\end{align}
By simple computations, we have,
\begin{align} \label{divergence beta}
  \notag   \operatorname{div} V &= \left(\frac{1-\beta}{N'}\right)^N \frac{X_1^{N- \beta(N-1)}}{\lvert x \rvert^N} \Bigg[N+ \frac{N}{1-\beta}\sum_{j=2}^m X_2 \cdots X_j\\
   & \hspace{100pt}+\frac{N'}{(1-\beta)^2}\sum_{j=2}^m \sum_{i=2}^j X_2^2 \cdots X_i^2X_{i+1} \cdots X_j \Bigg].
\end{align}
On the other hand we have,
\begin{align} \label{vector field}
 \lvert V \rvert^{N'} X_1^{\beta} = \left(\frac{1-\beta}{N'}\right)^N  \frac{X_1^{N- \beta(N-1)}}{\lvert x \rvert^N} \left[1+ \frac{1}{1-\beta}\sum_{j=2}^m X_2 \cdots X_j\right]^{N'}.
\end{align}
To estimate the term we will use Taylor's expansion,
\begin{align} \label{Taylor ineq1}
    \left[1+ \frac{1}{1-\beta}\sum_{j=2}^m X_2 \cdots X_j\right]^{N'}
    \le 1 &+ \frac{N'}{1-\beta} \sum_{j=2}^m X_2 \cdots X_j \notag\\ 
    & +\frac{N}{2(N-1)^2(1-\beta)^2} \left(\sum_{j=2}^m X_2 \cdots X_j\right)^2.
\end{align}
Thus, from  \eqref{vector field} we can deduce that,
\begin{align} \label{V eqn beta}
   \notag  (N-1) \lvert V \rvert^{N'} X_1^{\beta} &\le \left(\frac{1-\beta}{N'} \right)^N \frac{X_1^{N- \beta(N-1)}}{\lvert x \rvert^N} \Bigg[N-1 + \frac{N}{1-\beta} \sum_{j=2}^mX_2 \cdots X_j\\
  & \hspace{100pt}+\frac{N'}{2(1-\beta)^2}\left(\sum_{j=2}^m X_2 \cdots X_j\right)^2\Bigg].
\end{align}
Again as in the proof of Proposition \ref{series expansion without beta}, expanding the square in \eqref{V eqn beta}, we have,
\begin{align*}
    \operatorname{div} V - (N-1) \lvert V \rvert^{N'}X_1^{\beta} \ge \left(\frac{1-\beta}{N'} \right)^N  \frac{X_1^{N- \beta(N-1)}}{\lvert x \rvert^N} \Bigg[ 1+  \frac{N'}{2(1-\beta)^2} \sum_{j=2}^m X_2^2 \cdots X_j^2 \Bigg].
\end{align*}
Hence, the result follows from \eqref{div conclusion conclusion beta}.
\end{proof}
\begin{proposition} \label{p=k series expansion prop}
 Let $\lambda>0, k \ge 2$ and \eqref{p=k assumption} holds. Define $L = (\lambda + k -1)/ k.$ Then for all $u \in W^{1,k}_0(\Omega \setminus S, X_1^{\lambda} \, dx )$  we have the following inequality,
    \begin{align}
    \notag  \int_{\Omega} \lvert \nabla u \rvert^k X_1^{\lambda} -& L^k  \int_{\Omega} \frac{\lvert u \rvert^k}{d^k} X_1^{\lambda+k} \\
     &\hspace{50pt} \ge \frac{k-1}{2k} L^{k-2} \left(\sum_{i=2}^{\infty} \int_{\Omega} \frac{\lvert u \rvert^k}{d^k} X_1^{\lambda+k} X_2^2 \cdots X_i^2 \, dx\right).
    \end{align}
\end{proposition}
\begin{proof}
   For any $C^1$ vector field $V$ on $\Omega \setminus S$ and $u \in C^{\infty}_c(\Omega \setminus S)$, from \eqref{p=k eqn} we have,
    \begin{align*}
        \int_{\Omega} \lvert \nabla u \rvert^k X_1^{\lambda} \, dx \ge \int_{\Omega} \left(\operatorname{div} V-(k-1) \lvert V \rvert^{k'} X_1^{-\lambda/(k-1)}\right) \lvert u \rvert^{k} \, dx.
    \end{align*}
Consider the vector field,
    \begin{align*}
         V(x) = L^{k-1} \frac{\nabla d(x)}{d^{k-1}(x)} X_1^{\lambda + k-1} \left(1+ \frac{k-1}{kL}\sum_{i=2}^m X_2 \cdots X_i\right).
    \end{align*}
A straightforward computation yields,
\begin{align*}
    &\operatorname{div} V -(k-1) \lvert V \rvert^{k'} X_1^{-\lambda/(k-1)} \\
    & \qquad \ge L^k \frac{X_1^{\lambda + k}}{d^k} \left(k + \frac{k-1}{L}\sum_{i=2}^m X_2 \cdots X_i + \frac{k-1}{kL^2} \sum_{i=2}^m  \sum_{j=2}^i X_2^2 \cdots X_j^2 X_{j+1} \cdots X_i\right)\\
    & \hspace{140pt}- L^k \frac{X_1^{\lambda + k}}{d^k}(k-1)\left(1+ \frac{k-1}{kL}\sum_{i=2}^m X_2 \cdots X_i\right)^{k'}.
\end{align*}
To estimate the term in the right hand side, we will use Taylor's expansion,
\begin{align*}
    \left[1+ \frac{k-1}{kL}\sum_{i=2}^m X_2 \cdots X_i\right]^{k'} &\le 1 + \frac{1}{L} \sum_{i=2}^m X_2 \cdots X_i +\frac{1}{2k L^2} \left(\sum_{i=2}^m X_2 \cdots X_i\right)^2.
\end{align*}
Again as in the proof of Proposition \ref{series expansion without beta}, expanding the square in the last equation we can conclude that,
\begin{align*}
    \operatorname{div} V -(k-1) \lvert V \rvert^{k'} X_1^{-\lambda/(k-1)} \ge \frac{L^k}{d^k} X_1^{\lambda + k} \left(1+ \frac{k-1}{2L^2} \sum_{i=2}^m X_2^2 \cdots X_i^2\right).
\end{align*}
Hence, the result follows.
\end{proof}
\subsection{Proof of Theorem \ref{p not k thm}} The proof of \Cref{p not k thm} follows from Proposition \ref{p<k ineq hardy prop} and Proposition \ref{p<k series expansion} proved below.
\begin{proposition} \label{p<k series expansion}
 Let $\lambda > 0$ and $2 \le p<k$. Then for all $u \in W^{1,p}_0(\Omega \setminus S, X_1^{\lambda} \, dx )$  we have the following inequality,
    \begin{align}
       \notag  \int_{\Omega} \lvert \nabla u \rvert^p X_1^{\lambda} &- \left(\frac{k-p}{p}\right)^p  \int_{\Omega} \frac{\lvert u \rvert^p}{d^p} X_1^{\lambda}  \\
      &\hspace{40pt} \ge \frac{p-1}{2p} \left(\frac{k-p}{p}\right)^{p-2} \left(\sum_{i=1}^{\infty} \int_{\Omega} \frac{\lvert u \rvert^p}{d^p} X_1^{\lambda} X_1^2 \cdots X_i^2 \, dx\right).
    \end{align}
\end{proposition}
\begin{proof}
We denote $M= (k-p)/p.$ Note that we have,
    \begin{align*}
    \int_{\Omega} \lvert \nabla u \rvert^p X_1^{\lambda} \ge \int_{\Omega} \left[\operatorname{div} \left(X_1^{\lambda} V\right) - (p-1) \lvert V \rvert^{p'} X_1^{\lambda}\right] \lvert u \rvert^p \, dx.
\end{align*}
We set $Y(t) \coloneqq \sum_{i=1}^m X_1(t) \cdots X_i(t)$ and $Z(t) = \sum_{i=1}^m X_1^2(t) \cdots X_i^2(t).$ To prove the inequality we have to show that,
\begin{align*}
    \operatorname{div} \left(X_1^{\lambda} V\right) - (p-1) \lvert V \rvert^{p'} X_1^{\alpha} \ge \frac{M^p}{d^p} \left(1+ \frac{p-1}{2p M^2}Z\left(\frac{d}{R_{d,\Omega}}\right)\right).
\end{align*}
Consider the following vector field,
\begin{align*}
    V(x) = M^{p-1}\frac{\nabla d(x)}{d^{p-1}(x)} \left(1+ \frac{p-1}{pM}Y\left(\frac{d(x)}{R_{d,\Omega}}\right) \right).
\end{align*}
 Now, $\operatorname{div} \left(X_1^{\lambda} V\right)  = X_1^{\lambda} \operatorname{div} V + \langle V, \frac{\lambda}{d} X_1^{\lambda + 1} \nabla d \rangle \ge X_1^{\lambda} \operatorname{div} V.$ The computation of $\operatorname{div} V$ and the rest of the proof follows as in \cite[Proof of Theorem $A(1)$]{BarbatisSeries}, so we omit the details.
\end{proof}
\section{Best constants in the Improved Hardy-type inequalities} \label{Best constant in series}
In this section, we prove sharp constants in the improved Hardy-type inequalities and use these results to conclude the proofs of \Cref{series expansion}, \Cref{beta series expansion}, and \Cref{p=k series expansion}. The proofs are based on suitable choice of test functions together with some elementary inequalities. Similar techniques were used in \cite{FT, BarbatisSeries}.
\subsection{Proof of Theorem \ref{series expansion}} The proof of Theorem \ref{series expansion} follows from Proposition \ref{Improved sharp without beta} proved below in combination with Proposition \ref{beta=1} and Proposition \ref{series expansion without beta}. 
\begin{proposition} \label{Improved sharp without beta}
Let $\Omega$ be a bounded domain containing the origin. Suppose that for all $u \in W^{1,N}_0\left(\Omega, X_1^{-N+1}\right)$, there exists a positive $C_N$ such that the following inequality holds:
\begin{align*}
    I_{m -1}[u] \ge C_{N} \int_{\Omega} \frac{\lvert u \rvert ^N}{\lvert x \rvert^N} X_2^N X_1 X_3^2 \cdots X_k^{\rho} \, dx,
\end{align*}
where $I_m[u]$ is defined as in \eqref{finite series}.
Then $\rho \ge 2.$ Moreover, if $\rho = 2,$ then we have $ C_{N} \le \frac{1}{2} \left(\frac{N-1}{N}\right)^{N-1}.$    
\end{proposition}
\begin{proof}
     Since $\Omega$ is an open set contains the origin, so take $\delta > 0$ small enough such that $B(0,\delta) \subset \Omega.$ Now, consider $\eta \in C_c^{\infty}\left(B(0,\delta)\right)$ such that $0\le \eta \le 1$ and $\eta = 1$ on $B\left(0,\delta/2\right)$. Now, for given $\alpha_2, \cdots, \alpha_m>0$ small enough we set,
    \begin{align*}
        v(x) \coloneqq X_2^{\frac{-N+1 + \alpha_2}{N}} X_3^{\frac{-1+\alpha_3}{N}} \cdots X_m^{\frac{-1+\alpha_m}{N}} \quad \text{and} \quad u(x) \coloneqq \eta(x) v(x).
    \end{align*}
   Thus, we can derive,
   \begin{align*}
       \nabla v(x) = - \frac{X_1}{\lvert x \rvert} X_2^{\frac{1+\alpha_2}{N}}X_3^{\frac{-1+ \alpha_3}{N}} \cdots X_m^{\frac{-1+\alpha_m}{N}} \left[\frac{N-1}{N} + \frac{w(x)}{N}\right] \frac{x}{\lvert x \rvert},
   \end{align*}
 where $w(x) = - \alpha_2 + \sum_{j=3}^m (1-\alpha_{j})X_3 \cdots X_j.$ Now, since $\delta>0$ is small enough so $X_i$'s are also small so we can assume that $w(x)$ is small enough on $B(0, \delta)$. In particular $w(x)$ is much smaller than $(N-1)/N$ on $B(0,\delta)$. It can be checked that, $\int_{\Omega} \lvert \nabla u \rvert^N X_1^{-N+1} \, dx < \infty.$ Now, for $\varepsilon > 0$ small enough consider, $u_{\varepsilon}(x) \coloneqq X_1^{\varepsilon} u(x).$ Since $u_{\varepsilon}$ vanishes at the boundary of $\Omega \setminus \{0\}$, so, from Lemma \ref{contains lip functions} it follows that, $u_{\varepsilon} \in W^{1,N}_0 \left( \Omega \setminus \{0\}, X_1^{-N+1}\right)$. We will show that
 \begin{align} \label{convergence1}
     \int_{\Omega} \lvert \nabla u_{\varepsilon} - \nabla u\rvert^N X_1^{-N+1} \, dx \to 0.
 \end{align}
We have,
\begin{align*}
    \int_{\Omega} \lvert \nabla u_{\varepsilon} - \nabla u\rvert^N X_1^{-N+1} \, dx &\le c_N \varepsilon^N \int_{\Omega} \frac{\lvert u \rvert^N}{\lvert x \rvert^N} X_1^{N \varepsilon+1}\, dx\\
    & \qquad +  c_N \int_{\Omega} \lvert X_1^{\varepsilon} \nabla u - \nabla u\rvert^N X_1^{-N+1} \, dx \coloneqq A + B.
\end{align*}
By direct application of dominated convergence Theorem, it can be shown that $B \to 0$ as $\varepsilon \to 0.$ We show that $A \to 0$ as $\varepsilon \to 0.$ We can find $c_{\alpha_2}$ such that,
\begin{align*}
    \lvert v \rvert^N \le c_{\alpha_2} X_2^{-N + 1 +\frac{\alpha_2}{2}}.
\end{align*}
Therefore we have,
\begin{align*}
    A &= \Tilde{C_N} \varepsilon^N c_{\alpha_2} \int_{0}^{\delta} X_2^{-N + 1 + \frac{\alpha_2}{2}}(r) X_1^{N \varepsilon + 1}(r) r^{-1} \, dr\\
    &= \Tilde{C}\varepsilon^N \int_{c_\delta}^{\infty} z^{N - 1 - \frac{\alpha_2}{2}}  e^{-N \varepsilon z} \, dz \le \Tilde{C_1} \varepsilon^{\frac{\alpha_2}{2}} \int_{0}^{\infty} e^{-v} v^{N-1 -\frac{\alpha_2}{2}} \, dv \to 0 \quad \text{as} \quad \varepsilon \to 0.
\end{align*}
Thus, \eqref{convergence1} follows and together with \Cref{approximation}, we can conclude that, $u \in W^{1,N}_0\left( \Omega, X_1^{-N+1} \right).$
Now, using an elementary inequality \eqref{elt ineq1}, we deduce,
   \begin{align*}
       \int_{\Omega} \lvert \nabla u \rvert^N X_1^{-N+1} \, dx &\le \int_{\Omega} \eta^N \lvert \nabla v \rvert^N X_1^{-N+1} \, dx + c_N\int_{\Omega} \lvert \nabla \eta \rvert^N \lvert v \rvert^N X_1^{-N+1} \, dx\\
       & \qquad  + c_N \int_{\Omega} \lvert \nabla \eta \rvert \lvert \eta \rvert^{N-1} \lvert \nabla v \rvert^{N-1} \lvert v \rvert X_1^{-N+1} \, dx \coloneqq I_1 + I_2 + I_3.
   \end{align*} 
     We first show that $I_2, I_3 = O(1)$ uniformly with respect to $\alpha_i$'s for $i=2, \cdots m$.
    \begin{align} \label{I_2 without beta}
        I_2 = c\int_{0}^{\delta} X_2^{-N+1+\alpha_2}X_3^{-1+\alpha_3} \cdots X_m^{-1+ \alpha_m} X_1^{-N+1} \, r^{N-1} dr = O(1).
    \end{align}
Again using the fact that $w(x)$ is much smaller than $(N-1)/N$ on $B(0,\delta)$, we have,
\begin{align} \label{I_3 without beta}
    I_3 \le C_{N, \delta} \int_{0}^{\delta} X_2^{\alpha_2} X_3^{-1+ \alpha_3} \cdots X_m^{-1+\alpha_m} \, dr = O(1).
\end{align}
Since $w$ is small compared to $(N-1)/N$ by Taylor series expansion, we may have,
\begin{align*}
   \left \lvert \frac{N-1}{N} + \frac{w(x)}{N}\right \rvert^N  \le \left(\frac{N-1}{N}\right)^N + \left(\frac{N-1}{N}\right)^{N-1} w+ \frac{1}{2} \left(\frac{N-1}{N}\right)^{N-1} w^2 + c \lvert w \rvert^3
\end{align*}
Thus, using the above inequality, we have,
\begin{align} \label{expanding and getting J}
  \notag   \int_{\Omega} \lvert \nabla u \rvert^N& X_1^{-N+1} \, dx\\
  \notag   &\le  \int_{\Omega} \eta^N \lvert \nabla v \rvert^N X_1^{-N+1} \, dx + O(1)\\
  \notag   &= \int_{\Omega} \frac{\eta^N}{\lvert x \rvert^{N}} X_1 X_2 ^{1+ \alpha_2} X_3^{-1+ \alpha_2} \cdots X_m ^{-1+ \alpha_m} \left \lvert \frac{N-1}{N} + \frac{w(x)}{N}\right \rvert^N \, dx + O(1)\\
   \notag  &\le \int_{\Omega} \frac{\eta^N}{\lvert x \rvert^{N}} X_1 X_2 ^{1+ \alpha_2} X_3^{-1+ \alpha_2} \cdots X_m ^{-1+ \alpha_m} \Bigg[\left(\frac{N-1}{N}\right)^N + \left(\frac{N-1}{N}\right)^{N-1} w\\
    &\hspace{100pt}+ \frac{1}{2} \left(\frac{N-1}{N}\right)^{N-1} w^2 + c \lvert w \rvert^3 \Bigg] \, dx + O(1).
\end{align}
We now show that $\Lambda \coloneqq \int_{\Omega} \eta^N\lvert x \rvert^{-N} X_1 X_2 ^{1+ \alpha_2} X_3^{-1+ \alpha_2} \cdots X_m ^{-1+ \alpha_m} \lvert w \rvert^3 \, dx $ is uniformly bounded with respect to $\alpha_i$'s. Since $\lvert w \rvert^3 \le c \left( \alpha_2^3 + X_3^3\right)$, so we have,
\begin{align} \label{finiteness of L}
    \Lambda\le c \int_{\Omega} \frac{\alpha_2^3 \eta^N}{\lvert x \rvert^{N}}& X_1 X_2 ^{1+ \alpha_2} X_3^{-1+ \alpha_2} \cdots X_m ^{-1+ \alpha_m} \, dx\notag\\
    &+ c\int_{\Omega} \frac{\eta^N}{\lvert x \rvert^{N}} X_1 X_2 ^{1+ \alpha_2} X_3^{2+ \alpha_2} \cdots X_m ^{-1+ \alpha_m}  \, dx,
\end{align}
where the first and second integrals are bounded uniformly in $\alpha_i's$ due to factors $\alpha_2^3$ and $X_3^{2 + \alpha_2}$, respectively. Thus, by expanding $w^2$ in \eqref{expanding and getting J}, and collecting similar terms, we ended up with,
\begin{align*}
   {I}_{m-1}[u] \le \frac{1}{2} \left(\frac{N-1}{N}\right)^{N-1} J + O(1),
\end{align*}
where $J$ is defined as follows:
\begin{align} \label{error term J}
    J = A_m + \sum_{j=2}^m \left( \alpha_j^2 - 2 \alpha_j \right) A_j + 2 \sum_{j=2}^{m-1} \sum_{l = j+1}^m (1- \alpha_j) (1- \alpha_l) B_{jl},
\end{align}
with $A_j$ and $B_{jl}$ are defined as,
\begin{align*}
    A_j = \int_{\Omega} \frac{\eta^N}{\lvert x \rvert^{N}} X_1 X_2 ^{1+ \alpha_2} X_3^{1+ \alpha_3} \cdots X_j ^{1+ \alpha_m} X_{j+1}^{-1 + \alpha_{j+1}} \cdots X_m^{-1+ \alpha_m}\, dx \hspace{4pt} \text{for} \hspace{4pt} 2 \le j \le m,
\end{align*}
\begin{align*}
    B_{2l} = \int_{\Omega} \frac{\eta^N}{\lvert x \rvert^{N}} X_1 X_2 ^{1+ \alpha_2} X_3^{\alpha_3} \cdots X_l^{\alpha_l} X_{l+1}^{-1+ \alpha_{l+1}} \cdots X_m^{-1 +\alpha_m} \, dx \hspace{4pt} \text{for} \hspace{4pt} 3 \le j \le m,
\end{align*}
\begin{align*}
    B_{jl} = \int_{\Omega} \frac{\eta^N}{\lvert x \rvert^{N}} X_1 X_2 ^{1+ \alpha_2} X_3^{1+\alpha_3} \cdots X_j^{1+\alpha_j} X_{j+1}^{\alpha_{j+1}} \cdots X_l^{\alpha_l} X_{l+1}^{-1+ \alpha_{l+1}} \cdots X_m^{-1+ \alpha_m} \, dx ,
\end{align*}
for $ 3 \le j \le  l \le m$ with $B_{ll}=A_l.$ We want to take limit $\alpha_2 \to 0$ in \eqref{error term J}. Except $A_2$ and $B_{2l}$, all terms have finite limit as  $\alpha_2 \to 0$. We will combine those terms and then will tend $\alpha_2 \to 0.$ Now, by an application of integration by parts yields,
\begin{align} \label{identity log log}
   \notag \alpha_2 A_2 &= \alpha_2 \int_{\Omega} \frac{\eta^N}{\lvert x \rvert^{N}} X_1 X_2 ^{1+ \alpha_2} X_3^{-1+ \alpha_3} \cdots X_m ^{-1+ \alpha_m} \, dx\\
    \notag &= \omega_{N-1} \int_{0}^{\delta} \eta(r)^N \left(X_2^{\alpha_2}(r)\right)' X_3^{-1+ \alpha_3}(r) \cdots X_m ^{-1+ \alpha_m}(r) \, dr\\
   \notag &= \sum_{l=3}^m (1- \alpha_l) B_{2l} - N \omega_{N-1} \int_{0}^{\delta} \eta^{N-1} \eta' X_2^{\alpha_2} X_3^{-1+ \alpha_3} \cdots X_m^{-1 + \alpha_m} \, dr + O(1)\\
    &= \sum_{l=3}^m (1- \alpha_l) B_{2l} + O(1).
\end{align}
Similarly, we can show that,
\begin{align} \label{identity log back}
    \alpha_2 B_{2l} = - \sum_{k=3}^l \alpha_k B_{kl} + \sum_{k=l+1}^m (1- \alpha_k) B_{lk} + O(1).
\end{align}
Therefore, we have
\begin{align} \label{first error term without beta}
   \notag \left(\alpha_2^2 - 2 \alpha_2\right)A_2 &+ 2 \sum_{l=3}^m(1- \alpha_2)(1-\alpha_l) B_{2l} \\
   \notag &\stackrel{\eqref{identity log log}}{=}- \sum_{l=3}^m (1- \alpha_l) \alpha_2 B_{2l} + O(1)\\
  \notag  &\stackrel{\eqref{identity log back}}{=} - \sum_{l=3}^m (1- \alpha_l) \left(- \sum_{j=3}^l \alpha_j B_{jl} + \sum_{j=l+1}^m (1- \alpha_j) B_{lj}\right) + O(1)\\
    &= \sum_{j=3}^m \left(\alpha_j - \alpha_j^2\right) A_j + \sum_{j=3}^{m-1} \sum_{l=j+1}^m (2\alpha_j -1) (1- \alpha_l) B_{jl} + O(1).
\end{align}
 Substituting this into \eqref{error term J} and then letting $\alpha_2 \to 0$
we ended up with,
\begin{align}
    J = A_m - \sum_{j=3}^m\alpha_j A_j + \sum_{j=3}^{m-1} \sum_{l= j+1}^m (1- \alpha_l)B_{jl} + O(1), \quad (\alpha_2 = 0).
\end{align}
Similarly we want to let $\alpha_3 \to 0.$ Repeating the above process, we finally have,
\begin{align} \label{error term second J}
    J = (1- \alpha_m) A_m + O(1), \quad (\alpha_2= \alpha_3= \cdots =\alpha_{m-1} = 0).
\end{align}
Therefore, we have
\begin{align*}
    C_N \le \frac{1}{2} \left(\frac{N-1}{N}\right)^{N-1} \frac{(1- \alpha_m) A_m + O(1)}{\int_{\Omega} \frac{\eta^N}{\lvert x \rvert^{N}} X_1 X_2 X_3 \cdots X_m ^{\rho-1+ \alpha_m}}.
\end{align*}
 If $\rho  < 2,$ then letting $\alpha_m \to 2-\rho$ yields that the denominator goes to infinity while the numerator remains bounded, and hence $C_N=0.$ If $\rho=2,$ then the numerator becomes $A_m$ and $A_m \to \infty$ as $\alpha_m \to 0$ and hence we have,
   $C_N \le \frac{1}{2}\left(\frac{N-1}{N}\right)^{N-1}$.
\end{proof}
\subsection{Proof of Theorem \ref{beta series expansion}} The Proof of \Cref{beta series expansion} follows from the following Proposition \ref{beta improved series optimal} together with Proposition \ref{sharp hardy type1}, Proposition \ref{only series beta}, and \Cref{approximation}.

\begin{proposition} \label{beta improved series optimal}
Let $\Omega$ be a bounded domain containing the origin and $0<\beta<1$. Suppose that for all $u \in W_0^{1,N}\left(\Omega,X_1^{-\beta(N-1)}\right)$, there is some positive constant $C_{\beta,N}$ such that the following inequality holds:
\begin{align*}
    I_{m -1,\beta}[u] \ge C_{\beta,N} \int_{\Omega} \frac{\lvert u \rvert ^N}{\lvert x \rvert^N} X_1^{N - (N-1)\beta}X_2^2 \cdots X_k^{\rho} \, dx,
\end{align*}
where $I_{m,\beta}$ is defined as in \eqref{Improved beta hardy}.
Then $\rho \ge 2.$ Moreover, if $\rho = 2,$ then we have $ C_{\beta,N} \le \frac{N-1}{2N}H^{N-2},$ where $H=(1-\beta)/N'.$
\end{proposition}
\begin{proof}
 By \Cref{approximation}, we can assume that $u \in C_c^{\infty}\left(\Omega \setminus \{0\}\right).$ Since $\Omega$ is an open set contains the origin, so let $\delta > 0$ small enough such that $B(0,\delta) \subset \Omega.$ Now, consider $\eta \in C_c^{\infty}\left(B(0,\delta)\right)$ such that $0\le \eta \le 1$ and $\eta = 1$ on $B\left(0,\delta/2\right)$. Now, for given $\alpha_1, \cdots, \alpha_m>0$ small enough we set,
    \begin{align*}
        v(x) = X_1^{-H+\frac{\alpha_1}{N}} X_2^{\frac{-1+\alpha_2}{N}} \cdots X_m^{\frac{-1+\alpha_m}{N}} \quad \text{and} \quad u(x) = \eta(x) v(x).
    \end{align*}
Thus, we can derive,
   \begin{align*}
       \nabla v(x) = - \frac{x}{\lvert x \rvert^2} X_1^{1-H+\frac{ \alpha_1}{N} } X_2^{\frac{-1+ \alpha_2}{N}} \cdots X_m^{\frac{-1+\alpha_m}{N}} \left[H + \frac{w(x)}{N}\right] ,
   \end{align*}
 where $w(x) = - \alpha_1 + \sum_{j=2}^m (1-\alpha_{j})X_2 \cdots X_j.$ 
 Again as in the proof of Proposition \ref{Improved sharp without beta} by defining $u_{\varepsilon}(x) \coloneqq \lvert x \rvert^{\varepsilon} u(x)$ for small enough $\varepsilon>0$, we can similarly show that $u \in W^{1,N}_0\left( \Omega, X_1^{-\beta(N-1)} \right).$ Now, using an elementary inequality \eqref{elt ineq1} we deduce,
   \begin{align*}
       \int_{\Omega} \lvert \nabla u \rvert^N X_1^{-\beta(N-1)} \, dx &\le \int_{\Omega} \eta^N \lvert \nabla v \rvert^N X_1^{-\beta(N-1)} \, dx + c_N\int_{\Omega} \lvert \nabla \eta \rvert^N \lvert v \rvert^N X_1^{-\beta(N-1)} \, dx\\
       &\hspace{3pt}+ c_N \int_{\Omega} \lvert \nabla \eta \rvert \lvert \eta \rvert^{N-1} \lvert \nabla v \rvert^{N-1} \lvert v \rvert X_1^{-\beta(N-1)} \, dx \coloneqq I_1 + I_2 + I_3.
   \end{align*} 
Analogous to \eqref{I_2 without beta} and \eqref{I_3 without beta}, we can show that $I_2, I_3 = O(1)$ uniformly with respect to $\alpha_i$'s for $i=1,2, \cdots m$. Now, note that,
\begin{align*}
    I_1=\int_{\Omega} \eta^N \lvert \nabla v \rvert^N X_1^{-\beta(N-1)}=  \int_{\Omega} \frac{\eta^N}{\lvert x \rvert^{N}} X_1 ^{1+ \alpha_2} X_2^{-1+ \alpha_2} \cdots X_m ^{-1+ \alpha_m} \left \lvert H + \frac{w(x)}{N}\right \rvert^N.
\end{align*}
Since $w$ is small compared to $H,$ by Taylor series expansion, we may have,
\begin{align} \label{taylor improved const H}
   \left \lvert  H + \frac{w(x)}{N}\right \rvert^N \le H^N + H^{N-1}w + \frac{N-1}{2N} H^{N-2} w^2 + c \lvert w \rvert^3
\end{align}
Using this estimate \eqref{taylor improved const H}, we can bound $I_1$ by,
\begin{align} \label{estimate on first integral}
    I_1 \le J_1 + J_2 + J_3 + J_4 + O(1),
\end{align}
where
\begin{align*}
    J_1 &= H^N  \int_{\Omega} \frac{\eta^N}{\lvert x \rvert^{N}} X_1^{1+ \alpha_2} X_2^{-1+ \alpha_2} \cdots X_m ^{-1+ \alpha_m} \, dx = H^N  \int_{\Omega} \frac{\lvert u \rvert^N}{\lvert x \rvert^N} X_1^{N - \beta(N-1)} \, dx\\
    J_2 &= H^{N-1}\int_{\Omega} \frac{\eta^N}{\lvert x \rvert^{N}} X_1^{1+ \alpha_2} X_2^{-1+ \alpha_2} \cdots X_m ^{-1+ \alpha_m} w(x)\, dx\\
    J_3 &=  \frac{N-1}{2N} H^{N-2} \int_{\Omega} \frac{\eta^N}{\lvert x \rvert^{N}} X_1^{1+ \alpha_2} X_2^{-1+ \alpha_2} \cdots X_m ^{-1+ \alpha_m} w^2 \, dx\\
    J_4 &= c \int_{\Omega} \frac{\eta^N}{\lvert x \rvert^{N}} X_1^{1+ \alpha_2} X_2^{-1+ \alpha_2} \cdots X_m ^{-1+ \alpha_m} \lvert w(x) \rvert^3 \, dx.
\end{align*}
We now show that $J_2, J_4 = O(1)$ is uniformly bounded with respect to $\alpha_i$'s. Observe that $\lvert w \rvert^3 \le c \left( \alpha_1^3 + X_2^3\right)$. Similar argument as in \eqref{finiteness of L} shows that $J_4 = O(1).$
\par Let us now show that $J_2 = O(1).$ Let us first define some quantities.
\begin{align*}
    A_i &=\int_{\Omega} \frac{ \eta^N}{\lvert x \rvert^{N}} X_1 ^{1+ \alpha_1} X_2^{1+ \alpha_2} \cdots X_i ^{1+ \alpha_i} X_{i+1}^{-1+ \alpha_{i+1}} \cdots X_m^{-1+ \alpha_m} \quad \text{for} \quad 1\le i \le m\\
    B_{1l} &= \int_{\Omega} \frac{ \eta^N}{\lvert x \rvert^{N}} X_1 ^{1+ \alpha_1} X_2^{\alpha_2} \cdots X_l ^{\alpha_l} X_{l+1}^{-1+ \alpha_{l+1}} \cdots X_m^{-1+ \alpha_m} \, dx \quad \text{for} \quad 2\le l \le m\\
    B_{il} &= \int_{\Omega} \frac{ \eta^N}{\lvert x \rvert^{N}} X_1 ^{1+ \alpha_1} X_2^{1+ \alpha_2} \cdots X_i ^{1+ \alpha_i} X_{i+1}^{\alpha_{i+1}}X_{l+1}^{-1+ \alpha_{l+1}} \cdots X_m^{-1+ \alpha_m} \, dx,
\end{align*}
for $2\le i \le l \le m$ with $B_{ll} = A_l.$
Similar to \eqref{identity log log}, we can prove that,
\begin{align}
   \notag \alpha_1 A_1 = \sum_{i=2}^m (1- \alpha_i) B_{1l} + O(1).
\end{align}
Thus, it follows that, $J_2 = O(1).$ By similar arguments (see also \cite[Equation $3.9$, Equation $3.10$]{BarbatisSeries}), we can prove that, for $1 \le i \le m-1$ and $\alpha_1=\cdots = \alpha_{i-1}=0$,
\begin{align} 
    \alpha_i A_i &= \sum_{j=i+1}^m (1- \alpha_j) B_{ij} + O(1),\label{identities}\\  
    \alpha_i B_{ij} &= -\sum_{l=i+1}^{j} \alpha_l B_{lj} + \sum_{l=j+1}^m (1-\alpha_l) B_{jl} + O(1),\label{identities2}
\end{align}
Thus, using \eqref{estimate on first integral} we have,
\begin{align*}
    I_{m-1, \beta} [u] &\le   \frac{N-1}{2N}H^{N-2} \int_{\Omega} \frac{\eta^N}{\lvert x \rvert^{N}} X_1^{1+ \alpha_2} X_3^{-1+ \alpha_2} \cdots X_m ^{-1+ \alpha_m}\\
    &\hspace{6pt}\times \Bigg[ w^2 - \sum_{j=2}^{m-1} X_2^2 \cdots X_j^2\Bigg]\,dx + O(1) \eqqcolon  \frac{1}{2(1-\beta)}\left(\frac{1-\beta}{N'}\right)^{N-1} J + O(1).
\end{align*}
where $J$ is defined as follows:
\begin{align*}
    J &= \int_{\Omega} \frac{\eta^N}{\lvert x \rvert^{N}} X_1^{1+ \alpha_2} X_3^{-1+ \alpha_2} \cdots X_m ^{-1+ \alpha_m} \Bigg[\alpha_1^2 + \sum_{i=2}^m (1- \alpha_i)^2X_2^2 \cdots X_i^2\\
    & \hspace{40pt}-\sum_{i=2}^{m-1} X_2^2 \cdots X_i^2 - 2 \alpha_1 \sum_{l=2}^m (1-\alpha_l) X_2 \cdots X_l\\
    & \hspace{45pt}+ 2 \sum_{l=2}^{m-1}\sum_{i=l+1}^m (1- \alpha_i)(1-\alpha_l) X_2^2 \cdots X_i^2 X_{i+1}\cdots X_l\Bigg] \, dx\\
    &= \alpha_1^2 A_1 + A_m + \sum_{i=2}^m (\alpha_i^2 - 2 \alpha_i) A_i - 2 \alpha_1 \sum_{l=2}^m (1-\alpha_l) B_{1l}\\
    & \hspace{100pt}+ 2 \sum_{l=2}^{m-1}\sum_{i=l+1}^m (1- \alpha_i)(1-\alpha_l) B_{il} + O(1).
\end{align*}
We want to take limit $\alpha_1 \to 0$. Except $A_1$
and $B_{1l}$, all terms have finite limit as  $\alpha_1 \to 0$. By similar computations as in \eqref{first error term without beta}, using the identities \eqref{identities} and \eqref{identities2}, when we combine those terms involving $A_1$ and $B_{1l}$, we will have,
\begin{align*}
    \alpha_1^2 A_1 - 2 \alpha_1 \sum_{l=2}^m (1-\alpha_l) B_{1l}
    &=\sum_{i=2}^m \left(\alpha_i-\alpha_i^2\right) A_i \\
    & \hspace{50pt}+ \sum_{i=2}^{m-1} \sum_{l=i+1}^m (2\alpha_i -1) (1- \alpha_l) B_{il} + O(1).
\end{align*}
Substituting the above in $J$, and then letting $\alpha_1 \to 0$
we ended up with,
\begin{align}
    J = A_m - \sum_{i=2}^m\alpha_i A_i + \sum_{i=2}^{m-1} \sum_{l=i+1}^m (1- \alpha_l) B_{il} + O(1) \quad (\alpha_1 = 0),
\end{align}
where $O(1)$ is uniform with respect to $\alpha_1, \cdots, \alpha_m.$
 Repeating the above process, we finally have,
\begin{align} \label{p=k error term second J}
    J = (1- \alpha_m) A_m + O(1), \quad (\alpha_1= \alpha_2= \cdots =\alpha_{m-1} = 0).
\end{align}
The rest of the proof is similar to Proposition \ref{Improved sharp without beta}, so we omit the details.
\end{proof}
\subsection{Proof of Theorem \ref{p=k series expansion}} The proof of \Cref{p=k series expansion} follows from Proposition \ref{gen alpha improved series constant} proved in the following with Proposition \ref{p=k hardy ineq prop} and Proposition \ref{p=k series expansion prop}.

\begin{proposition} \label{gen alpha improved series constant}
    Let $\lambda>0$ and $\Omega$ be a domain in $\mathbb{R}^N$ with $R_{d,\Omega} < \infty$. Let $S$ be a piecewise closed smooth surface of codimension $k$ for $2\le k \le N$ and assume that $S \cap \Omega \neq \emptyset.$ For $k=N,$ we set $S=\{0\}$. Now, suppose that for some constant $C>0$ and $\gamma \in \mathbb{R}$ the following inequality holds:
    \begin{align}
        I_{m-1, \lambda}[u] \ge C  \int_{\Omega} \frac{\lvert u \rvert^k}{d^k} X_1^{\lambda+k} X_2^2 \cdots X_m^{\rho} \, dx,
    \end{align}
    where $I_{m,\lambda}[u]$ is defined as in \eqref{Improved hardy general d}. Then $\rho \ge 2.$ Moreover, if $\rho = 2$ then $C \le \frac{k-1}{2k} L^{k-2},$ for $L = (\lambda+k-1)/k$.
\end{proposition}
\begin{proof}
   Without loss of generality, assume that $0 \in \Omega \cap S$. Take $\delta > 0$ small enough so that $B(0,\delta) \subset \Omega.$ Consider $\eta \in C_c^{\infty}(B(0,\delta))$ such that $0 \le \eta \le 1$ and $\eta = 1$ on $B(0,\delta/2).$ Now, given $\alpha_1, \cdots, \alpha_m>0$ small enough, we set,
    \begin{align*}
        v(x) = X_1^{-L + \frac{\alpha_{1}}{k}} X_2^{\frac{-1+\alpha_2}{k}} \cdots X_m^{\frac{-1+\alpha_m}{k}} \quad \text{and} \quad u(x) = \eta(x) v(x).
    \end{align*}
Thus, we can derive,
   \begin{align*}
       \nabla v(x) = - \frac{\nabla d}{d} X_1^{1-L+\frac{ \alpha_1}{k} } X_2^{\frac{-1+ \alpha_2}{k}} \cdots X_m^{\frac{-1+\alpha_m}{k}} \left[L + \frac{w(x)}{k}\right] ,
   \end{align*}
 where $w(x) = - \alpha_1 + \sum_{i=2}^m (1-\alpha_{i})X_2 \cdots X_i.$ Analogous to the proof of Proposition \ref{Improved sharp without beta} and arguments in the proof of Proposition \ref{p<k ineq hardy prop}, by considering, $u_{\varepsilon}(x) \coloneqq d^{\varepsilon} u(x)$ for small enough $\varepsilon>0$, we can similarly show that $u \in W^{1,N}_0\left( \Omega \setminus S, X_1^{\lambda} \right).$
Now, using an elementary inequality \eqref{elt ineq1} we deduce,
   \begin{align*}
       \int_{\Omega} \lvert \nabla u \rvert^k X_1^{\lambda} \, dx &\le \int_{\Omega} \eta^k \lvert \nabla v \rvert^k X_1^{\lambda} \, dx + c_k \int_{\Omega} \lvert \nabla \eta \rvert^k \lvert v \rvert^k X_1^{\lambda} \, dx\\
       &\hspace{70pt}+ c_k \int_{\Omega} \lvert \nabla \eta \rvert \lvert \eta \rvert^{k-1} \lvert \nabla v \rvert^{k-1} \lvert v \rvert X_1^{\lambda} \, dx \coloneqq I_1 + I_2 + I_3.
   \end{align*} 
 Analogously to \eqref{I_2 without beta} and \eqref{I_3 without beta}, by using coarea formula, we can show that $I_2, I_3 = O(1)$ is uniformly bounded with respect to $\alpha_i$'s.
Now, note that,
\begin{align*}
    I_1=\int_{\Omega} \eta^k \lvert \nabla v \rvert^k X_1^{\lambda} \, dx =  \int_{\Omega} \frac{\eta^k}{d^k} X_1 ^{1+ \alpha_2} X_2^{-1+ \alpha_2} \cdots X_m ^{-1+ \alpha_m} \left \lvert L + \frac{w(x)}{k}\right \rvert^k \, dx.
\end{align*}
Since $w$ is small compared to $L$ by Taylor series expansion, we may have,
\begin{align} \label{taylor estimate L}
   \left \lvert  L + \frac{w(x)}{k}\right \rvert^k \le L^k + L^{k-1}w + \frac{k-1}{2k} L^{k-2} w^2 + c \lvert w \rvert^3.
\end{align}
Using this estimate \eqref{taylor improved const H}, we can bound $I_1$ by,
\begin{align*}
    I_1 \le J_1 + J_2 + J_3 + J_4 + O(1),
\end{align*}
where
\begin{align*}
    J_1 &= L^k  \int_{\Omega} \frac{\eta^k}{d^{k}} X_1^{1+ \alpha_1} X_2^{-1+ \alpha_2} \cdots X_m ^{-1+ \alpha_m} \, dx = L^k \int_{\Omega} \frac{\lvert u \rvert^k}{d^k} X_1^{\lambda+k} \, dx\\
    J_2 &= L^{k-1} \int_{\Omega} \frac{\eta^k}{d^{k}} X_1^{1+ \alpha_1} X_2^{-1+ \alpha_2} \cdots X_m ^{-1+ \alpha_m} w(x)\, dx\\
    J_3 &=  \frac{k-1}{2k} L^{k-2} \int_{\Omega} \frac{\eta^k}{d^k} X_1^{1+ \alpha_1} X_2^{-1+ \alpha_2} \cdots X_m ^{-1+ \alpha_m} w^2 \, dx\\
    J_4 &= c \int_{\Omega} \frac{\eta^k}{d^{k}} X_1^{1+ \alpha_1} X_2^{-1+ \alpha_2} \cdots X_m ^{-1+ \alpha_m} \lvert w(x) \rvert^3 \, dx.
\end{align*}
We now show that $J_2, J_4 = O(1)$ is uniformly bounded with respect to $\alpha_i$'s. Observe that $\lvert w \rvert^3 \le c \left( \alpha_1^3 + X_2^3\right)$. Thus, similar to \eqref{radial ltm function}, we can show that $J_4=O(1),$ see also \cite[Equation $3.16$]{Barbatis}.
Before proceeding further, let us fix some notations,
\begin{align*}
    A_i &=\int_{\Omega} \frac{ \eta^k}{d^{k}} X_1 ^{1+ \alpha_1} X_2^{1+ \alpha_2} \cdots X_i ^{1+ \alpha_i} X_{i+1}^{-1+ \alpha_{i+1}} \cdots X_m^{-1+ \alpha_m} \, dx \quad \text{for} \quad 1\le i \le m\\
    \Gamma_{1j} &= \int_{\Omega} \frac{ \eta^k}{d^{k}} X_1 ^{1+ \alpha_1} X_2^{\alpha_2} \cdots X_j ^{\alpha_j} X_{j+1}^{-1+ \alpha_{j+1}} \cdots X_m^{-1+ \alpha_m} \, dx \quad \text{for} \quad 2\le j \le m\\
    \Gamma_{ij} &= \int_{\Omega} \frac{ \eta^k}{d^{k}} X_1 ^{1+ \alpha_1} X_2^{1+ \alpha_2} \cdots X_i ^{1+ \alpha_i} X_{i+1}^{\alpha_{i+1}}X_{j+1}^{-1+ \alpha_{j+1}} \cdots X_m^{-1+ \alpha_m}, \hspace{3pt} \text{for} \hspace{2pt} 2\le i \le j \le m
\end{align*}
with $\Gamma_{ii} = A_{i}.$ For $1 \le i \le m-1$ and $\alpha_1=\cdots = \alpha_{i-1}=0$, we have,
\begin{align} 
    \alpha_i A_i &= \sum_{j=i+1}^m (1- \alpha_j) \Gamma_{ij} + O(1), \label{identities gen alpha}\\
    \alpha_i \Gamma_{ij} &= -\sum_{l=i+1}^{j} \alpha_l \Gamma_{lj} + \sum_{l=j+1}^m (1-\alpha_l) \Gamma_{jl} + O(1),   \label{identities2 gen alpha}
\end{align}
see the proof of Equations $(3.9),(3.10)$ in \cite{BarbatisSeries}. Although for the sake of completeness, we prove \eqref{identities gen alpha} and \eqref{identities2 gen alpha}. Following the step $3$ in the prrof of \cite[Proposition $3.1$]{BarbatisSeries}, a straightforward computation yields,
\begin{align} \label{divergen identity}
    \alpha_i d^{-k} X_1 \cdots X_{i-1}X_i^{1+ \alpha_i} = \operatorname{div}\left(d^{-k+1}X_i^{\alpha_i} \nabla d\right) - d^{-k}\left(d \Delta d + 1-k\right)X_i^{\alpha_i}.
\end{align}
Thus, 
\begin{align*}
    \alpha_i A_i &= \int_{\Omega} \eta^k \operatorname{div}\left(d^{-k+1}X_i^{\alpha_i} \nabla d\right) X_{i+1}^{-1+\alpha_{i+1}} \cdots X_m^{-1+\alpha_m} \, dx \\
    & \qquad\qquad- \int_{\Omega} \eta^k d^{-k}\left(d \Delta d + 1-k\right)X_i^{\alpha_i} X_{i+1}^{-1+\alpha_{i+1}} \cdots X_m^{-1+\alpha_m} \, dx \eqqcolon K_1 - K_2.
\end{align*}
By a direct consequence of \cite[Theorem $3.2$]{Ambro-son} we have, $d \Delta d + 1-k = O(d)$ as $d \to 0.$ Indeed, take $D = d^2/2$, then $\Delta D = 1 + d \Delta d$ and by \cite[Theorem $3.2$]{Ambro-son} we have, $\Delta D = k + O(d)$ as $d \to 0$ and hence combining these two we get $d \Delta d + 1-k = O(d)$ as $d \to 0.$ Thus, we can estimate $K_2$ by,
\begin{align*}
    c \int_{\Omega} \eta^k d^{-k+1}X_i^{\alpha_i} X_{i+1}^{-1+\alpha_{i+1}} \cdots X_m^{-1+\alpha_m} \, dx, \quad \text{for some constant} \hspace{3pt} c.
\end{align*}
which is uniformly bounded in $\alpha_i's.$ In case of $K_1,$ using integrating by parts yields,
\begin{align*}
    K_1 &= \int_{\Omega} \langle \nabla \eta^k, \nabla d \rangle d^{-k+1}X_i^{\alpha_i} X_{i+1}^{-1+\alpha_{i+1}} \cdots X_m^{-1+\alpha_m} \, dx\\
    & \hspace{90pt}- \int_{\Omega} \eta^k d^{-k+1}X_i^{\alpha_i} \left \langle \nabla d, \nabla \left(X_{i+1}^{-1+\alpha_{i+1}} \cdots X_m^{-1+\alpha_m} \right)\right\rangle \, dx,
\end{align*}
where the first term is $O(1)$ and the second term is $\sum_{j=i+1}^m (1- \alpha_j) \Gamma_{ij}.$ This proves \eqref{identities gen alpha}. Similarly, to prove \eqref{identities2 gen alpha}, we can again use \eqref{divergen identity} and proceed similarly, details are omitted.
Putting $i=1$ in \eqref{identities gen alpha} yields $J_2= O(1).$ Thus,
\begin{align*}
    I_{m-1}[u] &\le \frac{k-1}{2k}L^{k-2}\frac{\eta^k}{d^{k}} X_1^{1+ \alpha_1} X_2^{-1+ \alpha_2} \cdots X_m ^{-1+ \alpha_m} \left[ w^2 - \sum_{i=2}^m X_2^2 \cdots X_i^2\right] + O(1)\\
    & \eqqcolon \frac{k-1}{2k}L^{k-2} J + O(1),
\end{align*}
where we expand $w^2$ and collect the similar terms to define $J,$ more precisely,
\begin{align*}
    J &\coloneqq \alpha_1^2 A_1 + A_m + \sum_{i=2}^m (\alpha_i^2 - 2 \alpha_i) A_i - 2 \alpha_1 \sum_{l=2}^m (1-\alpha_l) \Gamma_{1l}\\
    &\hspace{120pt}+ 2 \sum_{l=2}^{m-1}\sum_{i=l+1}^m (1- \alpha_i)(1-\alpha_l) \Gamma_{il} + O(1).
\end{align*}
We  want to take the limit $\alpha_1 \to 0$ in the above. If we see separately, except for $A_1$ and $\Gamma_{1l}$, all terms have finite limit. Now, if we combine them, then using \eqref{identities gen alpha} and \eqref{identities2 gen alpha} (see also Step $7$ in the proof of \cite[Proposition $3.1$]{BarbatisSeries}), we can deduce that,
\begin{align*}
    \alpha_1^2 A_1 - 2 \alpha_1 \sum_{l=2}^m (1-\alpha_l) \Gamma_{1l}
    &=\sum_{i=2}^m \left(\alpha_i- \alpha_i^2\right) A_i \\
    & \qquad+ \sum_{i=2}^{m-1} \sum_{l=i+1}^m (2\alpha_i -1) (1- \alpha_l) \Gamma_{il} + O(1).
\end{align*}
Now, we can take the limit $\alpha_1 \to 0$ in $J$, letting $\alpha_1 \to 0$ we end up with,
\begin{align}
    J = A_m - \sum_{i=2}^m\alpha_i A_i + \sum_{i=2}^{m-1} \sum_{l= i+1}^m (1- \alpha_l) \Gamma_{il} + O(1) \quad (\alpha_1 = 0),
\end{align}
where $O(1)$ is uniform with respect to $\alpha_1, \cdots, \alpha_m.$
 Repeating the above process, we finally have,
\begin{align} \label{p=k error term second J second}
    J = (1- \alpha_m) A_m + O(1), \quad (\alpha_1= \alpha_2= \cdots =\alpha_{m-1} = 0).
\end{align}
The rest of the proof is similar to Proposition \ref{Improved sharp without beta}, so we omit the details.
\end{proof}

\section{Leray-Trudinger-type Improvement} \label{Log leray}
In this section, we will prove \Cref{Leray-weighted2} and \Cref{Leray weighted1}.
Let us first state the spherical harmonic decomposition of $L^2$ functions on $\mathbb{S}^{N-1}$. See \cite[Chapter IV, Section 2]{stein-weiss} and \cite[Chapter 1]{FengYuan} for details.
\begin{lemma}\label{decomposition lemma}
Let $u\in  L^{2}\left(\mathbb{B}_N\right)$, then employing the spherical coordinates $x= \left(r,\theta\right)$ in $\mathbb{B}_{N}$, we can decompose 
\begin{align}\label{spherical harmonic decomposotion}
u(x) = \sum_{k=0}^{\infty} u_k(r) f_k(\theta),
\end{align} 
 where $u_0(r) = \fint_{\mathbb{S}^{N-1}} u(r \theta) \, d \sigma\left(\theta\right)$, $f_0(\theta) = 1,$ $\left \lbrace f_k\right \rbrace_{k\in \mathbb{Z}^{+}}$ forms an orthonormal basis for $L^2(\mathbb{S}^{N-1})$ and satisfy 
 \begin{align*}
 -\Delta_{\mathbb{S}^{N-1}} f_k = k \left(k+N-2\right) f_k \mbox{ on } \mathbb{S}^{N-1}.
 \end{align*}
\end{lemma} 
We now recall a Lemma which will be used in proving \Cref{Leray weighted1} and \Cref{Leray-weighted2}. We refer to \cite[Lemma 4.2]{adimurthi2026leraytrudingertypeexponentialintegrability} for the proof, see also \cite[Equation (28)]{di2024optimal}.
\begin{lemma}\label{estimates of angular part}
  There exists a positive constant $c_N$ such that for any $v \in C^1_c(\mathbb{B}_N)$ we have,
    \begin{align}\label{estimates of angular part equation}
        \int_{\mathbb{B}_N} \left\lvert \nabla \left[(v-v_0)X_1^{-1+1/N}\right] \right \rvert^N \, dx \le c_N \int_{\mathbb{B}_N} \lvert \nabla v \rvert^N X_1^{-N+1} \, dx,
    \end{align}
    where $v_0$ is the spherical mean of $v$ \emph{i.e.} $v_0(r) = \fint_{\mathbb{S}^{N-1}} v(r \theta) \, d \sigma\left(\theta\right)$.
\end{lemma}
We also need the following Proposition, for a proof, see \cite[Proposition $3$]{seriesofoptimalleray} and see also \cite[Proposition $3.3$]{di2024optimal}.
\begin{proposition} \label{prop for radial version leray}
    Let $ q \ge 1$ and $\Gamma$ be the Gamma function. Then, there exists a positive constant $d_N$ such that for all $u \in C^1_{c, rad}\left(\mathbb{B}_N\right)$ and $j=1,2,\cdots,$ we have,
    \begin{align} \label{I_N radial estimate}
       &\left( \fint_{\mathbb{B}_N} \left(\lvert u \rvert^{2/N} X_1^{-1 + 1/N} X_{j+1}^{1/N}\right)^q \, dx\right)^{1/q}\notag \\
       & \qquad\le d_N \Gamma^{1/q}\left(1+ \frac{q(N-1)}{N}\right) \left(\fint_{\mathbb{B}_N} \lvert x \rvert^{2-N} \lvert \nabla u \rvert^2 X_1^{-1} \cdots X_j^{-1} \, dx\right)^{1/N}.
    \end{align}
\end{proposition}
\subsection{Proof of Theorem \ref{Leray-weighted2}} We are now ready to prove \Cref{Leray-weighted2}. Let us first prove the following Lemma.
\begin{lemma} \label{first lemma for proof weighted leray}
Let $u \in C^1_c\left(\Omega \setminus \{0\}\right)$ and set $v \coloneqq u X_2^{1-\frac{1}{N}}$. Recall that $ \Tilde{I}_N[u]$ as defined in \eqref{weighted Leray energy1}. Then we have the following estimates,
\begin{enumerate} 
\item[(i)] For some positive constant $\zeta_N$, there holds,
\begin{align} \label{estimate of weighted leray2}
        \int_{\Omega} \lvert x \rvert^{2-N} \lvert v \rvert^{N-2} \lvert \nabla v \rvert^2 X_2^{-1} X_1^{-1} \, dx \le \zeta_{N}  \Tilde{I}_N[u].
     \end{align}
\item[(ii)] We also have,
\begin{align} \label{second estimate on weighted leray2}
     \int_{\Omega} \lvert \nabla v \rvert^{N} X_2^{-N+1} X_1^{-N+1} \le \left(2^{N-1} -1\right) \Tilde{I}_N[u].
\end{align}
    \end{enumerate}
\end{lemma}
\begin{proof}
     Let us first prove \eqref{estimate of weighted leray2}. We have $u = v X_2^{-1+\frac{1}{N}}$. By using the inequality (see \cite[Lemma $3.1$ $(ii).b$ ]{Barbatis}): for some positive constant $\mu_N$, there holds, \[ \lvert b-a \rvert^N \ge \lvert a \rvert^N  + \mu_N\lvert a \rvert^{N-2} \lvert b \rvert^2 - N \lvert a \rvert^{N-2} \langle a, b\rangle \quad \text{for all} \hspace{3pt} a,b \in \mathbb{R}^N, \] 
    we can deduce that,
    \begin{align*}
        \lvert \nabla u \rvert^N &= \left \lvert X_2^{-1+\frac{1}{N}}\nabla v -\left(-1+\frac{1}{N}\right) X_2^{\frac{1}{N}} X_1\frac{x}{\lvert x \rvert^2} v \right\rvert^N\\
        &\begin{aligned} \ge \left \lvert \frac{N-1}{N}\right \rvert^{N}\frac{\lvert v \rvert^N}{\lvert x \rvert^N} X_2 X_1^{N} &+\frac{1}{\zeta_{N}} \lvert x \rvert^{2-N} \lvert v \rvert^{N-2} \lvert \nabla v \rvert^2 X_2^{-1} X_1^{N-2}\\
        & -N \left \lvert \frac{N-1}{N}\right \rvert^{N-1} X_1^{N-1} \left \langle \nabla \left( \lvert v \rvert^{N} \right), \frac{x}{\lvert x \rvert^N}  \right \rangle.
        \end{aligned}
    \end{align*}
   Multiplying both sides by $X_1^{-N+1}$ and doing the integration over $\Omega$ yields,
    \begin{align*}
        \int_{\Omega} &\lvert \nabla u \rvert^N X_1^{-N+1)} \, dx - \left ( \frac{N-1}{N}\right )^N \int_{\Omega} \frac{\lvert u \rvert ^N}{\lvert x \rvert^N} X_2^{N} X_1 \, dx \\
        & \qquad \ge \frac{1}{\zeta_{N}} \int_{\Omega} \lvert x \rvert^{2-N} \lvert v \rvert^{N-2} \lvert \nabla v \rvert^2 X_2^{-1} X_1^{-1} \, dx + \Tilde{C}_N \int_{\Omega} \lvert v \rvert^N \operatorname{div}\left(\frac{x}{\lvert x \rvert^N}\right)\, dx.
    \end{align*}
This completes the proof, since $\operatorname{div} \left(\frac{x}{\lvert x \rvert^N}\right) = 0$ on $\Omega \setminus \{0\}.$
 \par Similarly, by using the inequality, (see \cite{Peter}):
    \[\lvert b-a \rvert^N \ge \lvert a \rvert^N  + \frac{1}{2^{N-1}-1} \lvert b \rvert^N - N \lvert a \rvert^{N-2} \langle a, b\rangle \quad \text{for} \quad a,b \in \mathbb{R}^N,\]
    we can derive that $\int_{\Omega} \lvert \nabla v \rvert^{N} X_1^{-N+1} \le \left(2^{N-1}-1\right) \Tilde{I}_N[u].$
\end{proof}
\begin{proof} [Proof of \Cref{Leray-weighted2}] \label{proof of leray-weighted2}
  By \Cref{approximation} we can assume that $u \in C^1_c(\Omega \setminus \{0\})$. Extending by zero outside we may assume $u \in C^1_c\left(B\left(0,R_{\Omega}\right)\right) \setminus \{0\}$ and by scaling, we may assume that $R_{\Omega} = 1.$ Set $v = u X_2^{1 - 1/N}$. By using \Cref{decomposition lemma}, we write $v(x) = \sum_{k=0} v_k(r) f_k(\theta)$. Now, set $u_0= v_0 X_2^{-1+1/N}$.  Following Trudinger's technique \cite{Trudinger}, it's enough to establish the following inequality,
    \begin{align}
        \left(\int_{\mathbb{B}_N} \left(\lvert u \rvert X_1^{-1/N'}X_2^{1/N'}X_3^{1/N}\right)^q \right)^{1/q} \le c_N q^{1-1/N} \Tilde{I}_N[u], \quad \text{for} \quad q>N.
    \end{align}
    Then choosing $q=mN'$ and summing over $m=1,2,\cdots$, we will have the inequality \eqref{weighted leray-moser}. Note that we have,
    \begin{align*}
        &\left(\int_{\mathbb{B}_N} \left(\lvert u \rvert X_1^{-1/N'}X_2^{1/N'}X_3^{1/N}\right)^q \right)^{1/q} \\&\le\left(\int_{\mathbb{B}_N} \left(\lvert u -u_0\rvert X_1^{-1/N'}X_2^{1/N'}\right)^q \right)^{1/q}+\left(\int_{\mathbb{B}_N} \left(\lvert u_0 \rvert X_1^{-1/N'}X_2^{1/N'}X_3^{1/N}\right)^q \right)^{1/q}\\
        &\le \left(\int_{\mathbb{B}_N} \left(\lvert v -v_0\rvert X_1^{-1/N'} \right)^q \right)^{1/q} +\left(\int_{\mathbb{B}_N} \left(\lvert v_0 \rvert X_1^{-1/N'}X_3^{1/N}\right)^q \right)^{1/q} \coloneqq I_1 + I_2,
    \end{align*}
   where in the first step we have used the fact that $X_3 \le 1.$ As, $\left( v- v_0 \right) X_1^{-1 + 1/N} \in C_c^1\left(\mathbb{B}_N \setminus \{0\}\right)$, so from the proof of \cite[Theorem $1$]{Trudinger} for $q>N$, we have,
\begin{align*}
   \left(\int_{\mathbb{B}_N} \left(\lvert v- v_0 \rvert X_1^{-1 + \frac{1}{N}}\right)^q \, dx\right)^{1/q} &\le c_{N} q^{1-\frac{1}{N}}\left(\int_{\mathbb{B}_N} \left\lvert \nabla \left[(v-v_0)X_1^{-1 + \frac{1}{N}} \right] \right\rvert^N\right)^{1/N}\\
    &\stackrel{\eqref{estimates of angular part equation}}{\le}  c(N) q^{1-\frac{1}{N}} \int_{\mathbb{B}_N} \lvert \nabla v \rvert^N X_1^{-N+1} \, dx\\
     &\stackrel{\left(X_2 \leq 1\right)}{\le}  c(N) q^{1-\frac{1}{N}} \int_{\mathbb{B}_N} \lvert \nabla v \rvert^N X_1^{-N+1} X_2^{-N+1} \, dx \\
    & \stackrel{\eqref{second estimate on weighted leray2}}{\le} c_{N}  q^{1-\frac{1}{N}} \Tilde{I}_{N} [u].
\end{align*}
Now, define, $w \coloneqq \lvert v \rvert^{N/2}$. Then we have,
\begin{align}
    \lvert v_0(r) \rvert^{N/2} \le \left \lvert \fint_{\mathbb{S}^{N-1}} v(x) \, d \sigma(x)\right \rvert^{N/2} \le \fint_{\mathbb{S}^{N-1}} \lvert v(x) \rvert^{N/2} \, d \sigma(x)= w_0(r).
\end{align}
Thus, using above and from Proposition \ref{prop for radial version leray} we can have,
\begin{align*}
  I_2 \le  & \left( \fint_{\mathbb{B}_N} \left(w_0^{2/N} X_1^{-1 + 1/N} X_{3}^{1/N}\right)^q \, dx\right)^{1/q} \\
    & \le d_N \Gamma^{1/q}\left(1+ \frac{q(N-1)}{N}\right) \left(\fint_{\mathbb{B}_N} \lvert x \rvert^{2-N} \lvert \nabla w_0 \rvert^2 X_1^{-1} X_2^{-1} \, dx\right)^{1/N},
\end{align*}
where for large $q$, we use Stirling's approximation $\Gamma(1+x) \sim x^xe^{-x}$. We now show that, $\int_{\mathbb{B}_N} \lvert x \rvert^{2-N} \lvert \nabla w_0 \rvert^2 X_2^{-1} X_1^{-1} \le \int_{\mathbb{B}_N} \lvert x \rvert^{2-N} \lvert \nabla w \rvert^2 X_2^{-1} X_1^{-1}.$ The proof follows similarly as the proof of \cite[Equation $3.10$]{seriesofoptimalleray}. For the completeness, we prove this. First, set $f(x) \coloneqq \phi(\lvert x \rvert) w(x)$, where $\phi(\lvert x \rvert) = \lvert x \rvert^{1-N/2} X_1^{-1/2} X_2^{-1/2}.$ Now the spherical mean of $f$ is given by $f_0(r) = \phi(r) w_0(r).$ Thus,
\begin{align*}
    \int_{\mathbb{B}_N} \lvert \nabla f_0 \rvert^2 \, dx &= - \int_{\mathbb{B}_N} \phi \Delta \phi w_0^2 \, dx + \int_{\mathbb{B}_N} \phi^2  \lvert \nabla w_0 \rvert^2 \, dx\\
    &= - \int_{\mathbb{B}_N} \frac{\Delta \phi}{\phi} f_0^2 \, dx + \int_{\mathbb{B}_N} \phi^2  \lvert \nabla w_0 \rvert^2 \, dx\\
    &= \frac{(N-2)^2}{4 \lvert x \rvert^2}\int_{\mathbb{B}_N} \frac{f_0^2}{\lvert x \rvert^2} + \frac{1}{4} \sum_{i=1}^2 \int_{\mathbb{B}_N} \frac{X_1^2 X_i^2}{\lvert x \rvert^2} f_0^2 \, dx\\
    & \hspace{100pt}+ \int_{\mathbb{B}_N} \lvert x \rvert^{2-N} X_1^{-1} X_2^{-1} \lvert \nabla w_0 \rvert^2 \, dx.
\end{align*}
Therefore by using \cite[Equation $7.6$]{FT}, we have,
\begin{align}
\int_{\mathbb{B}_N} &\lvert x \rvert^{2-N} X_1^{-1} X_2^{-1} \lvert \nabla w_0 \rvert^2 \, dx\\
\notag  &=  \int_{\mathbb{B}_N} \lvert \nabla f_0 \rvert^2 \, dx - \frac{(N-2)^2}{4 \lvert x \rvert^2}\int_{\mathbb{B}_N} \frac{f_0^2}{\lvert x \rvert^2} -\frac{1}{4} \sum_{i=1}^2 \int_{\mathbb{B}_N} \frac{X_1^2 X_i^2}{\lvert x \rvert^2} f_0^2 \, dx\\
   \notag &\le \int_{\mathbb{B}_N} \lvert \nabla f \rvert^2 \, dx - \frac{(N-2)^2}{4 \lvert x \rvert^2}\int_{\mathbb{B}_N} \frac{f^2}{\lvert x \rvert^2}\-\frac{1}{4} \sum_{i=1}^2 \int_{\mathbb{B}_N} \frac{X_1^2 X_i^2}{\lvert x \rvert^2} f^2 \, dx\\
   \notag &=  \int_{\mathbb{B}_N} \lvert x \rvert^{2-N} X_1^{-1} X_2^{-1} \lvert \nabla w \rvert^2 \, dx \\
  \notag  &\le c_{N} \int_{\Omega} \lvert x \rvert^{2-N} \lvert v \rvert^{N-2} \lvert \nabla v \rvert^2 X_2^{-1} X_1^{-1} \, dx\stackrel{\eqref{estimate of weighted leray2}}{\le } \Tilde{c_{N}} \Tilde{I}_N[u],
\end{align}
where the second last inequality follows from the fact that $w = \lvert v \rvert^{N/2}.$
\par We still need to prove the optimality of the exponent. Let $0 < \gamma < 1/N.$ If possible let for any $u \in W^{1,N}_{0}\left(\Omega \setminus \{0\}, X_1^{-N+1}\right)$ with $\Tilde{I}_{N}[u] \le 1,$ there holds,
\begin{align} \label{gamma leray trudinger}
    \int_{\mathbb{B}_N} e^{\alpha \lvert u  X_1^{\frac{-1}{N'}} X_2^{\frac{1}{N'}} X_3^{\gamma}\rvert^{N'}} \le C_{N}, \quad \text{for all} \quad \alpha < A_N.
\end{align}
Since $0 < \gamma < 1/N,$ so there exists $1 < \rho <2$ such that $N \gamma + \rho < 2.$ We show that,
\begin{align} \label{}
    \int_{\Omega} \frac{\lvert u \rvert ^N}{\lvert x \rvert^N} X_2^N X_1 X_3^{N \gamma + \rho} \le \Tilde{C}_N \Tilde{I}_N[u],
\end{align}
which can happen only if $N \gamma + \rho \ge 2$ due to Proposition \ref{Improved sharp without beta}.
We have,
\begin{align*}
   \int_{\mathbb{B}_N} &\frac{\lvert u \rvert ^N}{\lvert x \rvert^N} X_2^N X_1 X_3^{N \gamma + \rho}\\
  &= \left(\frac{\alpha}{N-1} \right)^{1-N}\int_{\mathbb{B}_N} \left( \left(\frac{\alpha}{N-1} \right)^{N-1} \frac{\lvert u \rvert^N}{X_1^{N-1}} X_2^{N -1} X_3^{N \gamma} \right)\left(\frac{X_1^{N} X_2X_3^{\rho}}{\lvert x \rvert^N}\right) \, dx\\
   & \le \left(\frac{\alpha}{N-1} \right)^{1-N} 2^{N-2} \left(
\int_{\mathbb{B}_N} e^{\alpha \lvert u  X_1^{\frac{-1}{N'}} X_2^{\frac{1}{N'}} X_3^{\frac{1}{N}}\rvert^{N'}} + K \right),
\end{align*}
where, $K = \int_{\mathbb{B}_N} \left(1+\frac{X_1^{N} X_2 X_3^{\rho}}{\lvert x \rvert^N}\right) \left(\ln \left(1 + \left(\frac{X_1^{N} X_2X_3^{\rho}}{\lvert x \rvert^N}\right)^{1/(N-1)}\right)\right)^{N-1}$ and the last inequality follows from the following version of Young's inequality: for any $a,b \ge 0,$
\begin{equation} \label{original young's ineq}
    ab \le 2^{N-2} \left[ e^{(N-1) a ^{1/(N-1)}} + 2^{N-2}(1+b) \left(\ln \left(1+b^{1/(N-1)}\right)\right)^{N-1}\right].
\end{equation}
For a proof of \eqref{original young's ineq} one can refer to \cite[Equation $4.5$]{PsaradakisJFA}. A straightforward computation shows that $K$ is finite, see also \cite[Equation $4.5$]{seriesofoptimalleray}. This contradicts \eqref{gamma leray trudinger}.
\end{proof}
\subsection{Proof of Theorem \ref{Leray weighted1}} Finally, we prove \Cref{Leray weighted1}. First recall the quantity $I_{N,w_{\beta}}[u]$ as defined in \eqref{beta weighted leray energy}. Analogously, to the proof of Lemma \ref{first lemma for proof weighted leray}, we can prove the following Lemma.
\begin{lemma} \label{elementary leray estimate1}
    Let $u \in C^1_c\left(\Omega \setminus \{0\}\right), \beta \neq 1$ and set $v \coloneqq u X_1^{(1-\beta)/N'}$. Then we have,
\begin{enumerate} 
\item[(i)] For some positive constant $\zeta_{N,\beta},$ there holds,
\begin{align} \label{estimate of weighted leray1}
        \int_{\Omega} \lvert x \rvert^{2-N} \lvert v \rvert^{N-2} \lvert \nabla v \rvert^2 X_1^{-1} \, dx \le \zeta_{N,\beta}  I_{N,w_{\beta}}[u].
        \end{align}
\item[(ii)] We also have,
\begin{align} \label{second estimate of weighted leray1}
     \int_{\Omega} \lvert \nabla v \rvert^{N} X_1^{-N+1} \le \left(2^{N-1} -1\right) I_{N,w_{\beta}}[u].
\end{align}
    \end{enumerate}
\end{lemma}
\begin{proof} [Proof of \Cref{Leray weighted1}]
    Let $u \in C^1_c(\Omega \setminus \{0\})$, extending by zero outside we may assume $u \in C^1_c\left(B\left(0,R_{\Omega}\right)\right) \setminus \{0\}$ and by scaling, we may assume that $R_{\Omega} = 1.$ Set $v = u X_1^{(1-\beta)/N'}$. By \Cref{decomposition lemma}, we write $v(x) = \sum_{k=0} v_k(r) f_k(\theta)$. Now, set $u_0= v_0 X_1^{(\beta-1)/N'}$. Again we will use Trudinger's technique \cite{Trudinger} and show that,
    \begin{align}
        \left(\int_{\mathbb{B}_N} \left(\lvert u \rvert X_1^{-\beta/N'}X_3^{1/N}\right)^q \right)^{1/q} \le c_N q^{1-1/N} I_{N,w_{\beta}}[u], \quad \text{for} \quad q>N.
    \end{align}
    Note that we have,
    \begin{align*}
        &\left(\int_{\mathbb{B}_N} \left(\lvert u \rvert X_1^{-\beta/N'}X_2^{1/N}\right)^q \right)^{1/q}\\
        &\le \left(\int_{\mathbb{B}_N} \left(\lvert u -u_0\rvert X_1^{-\beta/N'}X_2^{1/N}\right)^q \right)^{1/q}+\left(\int_{\mathbb{B}_N} \left(\lvert u_0 \rvert X_1^{-\beta/N'}X_2^{1/N}\right)^q \right)^{1/q}\\
        &\le \left(\int_{\mathbb{B}_N} \left(\lvert v -v_0\rvert X_1^{-1/N'}X_2^{1/N}\right)^q \right)^{1/q} +\left(\int_{\mathbb{B}_N}
        \left(\lvert v_0 \rvert X_1^{-1/N'}X_2^{1/N}\right)^q \right)^{1/q} \coloneqq I_1 + I_2,
    \end{align*}
   To estimate $I_1,$ first note that $X_2 \le 1$ and as $\left( v- v_0 \right) X_1^{-1 + \frac{1}{N}} \in C_c^1\left(\mathbb{B}_N \setminus \{0\}\right)$ so from the proof of \cite[Theorem $1$]{Trudinger} for $q>N$, we have,
\begin{align*}
    \left(\int_{\mathbb{B}_N} \left(\lvert v- v_0 \rvert X_1^{-1 + \frac{1}{N}}\right)^q \, dx\right)^{1/q} &\le c_{1}(N) q^{1-\frac{1}{N}}\left(\int_{\mathbb{B}_N} \left\lvert \nabla \left[(v-v_0)X_1^{-1 + \frac{1}{N}} \right] \right\rvert^N\right)^{1/N}\\
    &\stackrel{\eqref{estimates of angular part equation}}{\le}  c(N) q^{1-\frac{1}{N}} \int_{\mathbb{B}_N} \lvert \nabla v \rvert^N X_1^{-N+1} \, dx\\
    & \stackrel{\eqref{second estimate of weighted leray1}}{\le} c_{N,\beta}  q^{1-\frac{1}{N}} I_{N, w_{\beta}} [u].
\end{align*}
Again as in the proof of \eqref{Leray-weighted2}, we set  $w \coloneqq \lvert v \rvert^{N/2}$. Then we have,
\begin{align}
    \lvert v_0(r) \rvert^{N/2} \le \left \lvert \fint_{\mathbb{S}^{N-1}} v(x) \, d \sigma(x)\right \rvert^{N/2} \le \fint_{\mathbb{S}^{N-1}} \lvert v(x) \rvert^{N/2} \, d \sigma(x)= w_0(r).
\end{align} Note that by using above, we have,
\begin{align*}
   I_2 &\le \left( \int_{\mathbb{B}_N} \left(\lvert w_0 \rvert^{2/N} X_1^{-1/N'} X_2^{1/N}\right)^q \, dx \right)^{1/q} \\
   &\stackrel{\eqref{I_N radial estimate}}{\le} d_N  \Gamma^{1/q}\left(1+ \frac{q(N-1)}{N}\right) \left(\int_{\mathbb{B}_N} \lvert x \rvert^{2-N} \lvert \nabla w_0 \rvert^2 X_1^{-1} \, dx\right)^{1/N},
\end{align*}
Now for large $q$, we use Stirling's approximation $\Gamma(1+x) \sim x^xe^{-x}$. Similar to the proof of \Cref{Leray-weighted2}, by setting  $f(x) \coloneqq  \lvert x \rvert^{1-N/2} X_1^{-1/2} w(x),$ (see also the proof of \cite[Equation $27$]{di2024optimal}) we can show that,
\begin{align*}
    \int_{\mathbb{B}_N} \lvert x \rvert^{2-N} \lvert \nabla w_0 \rvert^2 X_1^{-1} \, dx &\le \int_{\mathbb{B}_N} \lvert x \rvert^{2-N} \lvert v \rvert^{N-2} \lvert \nabla v \rvert^2 X_1^{-1} \, dx.
\end{align*}
Now, because of \eqref{estimate of weighted leray1} we finally have,
\begin{align*}
      \left( \int_{\mathbb{B}_N} \left(\lvert v_0 \rvert X_1^{-1/N'} X_2^{1/N}\right)^q \, dx \right)^{1/q} \le c(N,\beta) q^{1-1/N} I_{N,w_{\beta}}[u].
\end{align*}
\par Now, we will prove the optimality of the exponent for $\beta \in (-\infty,1)$. Let $0 < \gamma < 1/N.$ We will prove by contradiction. Assume that, for any $u \in W^{1,N}_{0}\left(\Omega \setminus \{0\}, w_{\beta}\right)$ with $I_{N, w_{\beta}}[u] \le 1,$ there holds,
\begin{align} \label{not true leray}
    \int_{\mathbb{B}_N} e^{\alpha \lvert u  X_1^{\frac{-\beta}{N'}} X_2^{\gamma}\rvert^{N'}} \le C_{N\beta}, \quad \text{for all} \quad \alpha < A_{N,\beta}.
\end{align}
Since $0 < \gamma < 1/N,$ so there exists $1 < \rho <2$ such that $N \gamma + \rho < 2.$ We will prove,
\begin{align} \label{gamma leray-Trui2}
    \int_{\Omega} \frac{\lvert u \rvert ^N}{\lvert x \rvert^N} X_1^{N - (N-1) \beta}X_2^{N \gamma + \rho} \le c_{N,\beta},
\end{align}
which is valid only if $N \gamma + \rho \ge 2,$ for $\beta \in (0,1)$ see Proposition \ref{beta improved series optimal} and in the case of $\beta \in (-\infty,0)$ it follows from Proposition \ref{gen alpha improved series constant}.
Analogous to the proof of \Cref{Leray-weighted2}, we can have,
\begin{align*}
    \int_{\mathbb{B}_N} &\frac{\lvert u \rvert ^N}{\lvert x \rvert^N} X_1^{N - (N-1) \beta}X_2^{N \gamma + \rho} \, dx\\
    & =  \left(\frac{\alpha}{N-1} \right)^{1-N} \int_{\mathbb{B}_N} \left( \left(\frac{\alpha}{N-1} \right)^{N-1} \lvert u \rvert^N X_1^{-\beta(N-1)} X_2^{N \gamma} \right) \left(\frac{X_1^{N} X_2^{\rho}}{\lvert x \rvert^N}\right) \, dx\\
    &\le \left(\frac{\alpha}{N-1} \right)^{1-N} 2^{N-2} \left(
\int_{\mathbb{B}_N} e^{\alpha \lvert u  X_1^{\frac{-\beta}{N'}}  X_2^{\gamma}\rvert^{N'}} + K \right),
\end{align*}
where, $K = \int_{\mathbb{B}_N} \left(1+\frac{X_1^{N} X_2^{\rho}}{\lvert x \rvert^N}\right) \left(\ln \left(1 + \left(\frac{X_1^{N} X_2^{\rho}}{\lvert x \rvert^N}\right)^{1/(N-1)}\right)\right)^{N-1} \, dx$. A simple computations yields the finiteness of $K$ (see also \cite{PsaradakisJFA}) and this contradicts \eqref{not true leray}.
\end{proof}
\section{Interpolation with Leray-Trudinger-type Inequalities} \label{interpolation}
An interpolation of \eqref{Hardy2} and the Moser-Trudinger inequality \cite{Moser} was established by Adimurthi and Sandeep \cite{Adi-San} (see \cite{Adi-Yang} for generalization in the whole of $\mathbb{R}^N$). More precisely, their result states as: for $\alpha > 0$ and $\sigma \in [0,N)$ the following holds,
    \begin{align} \label{Adi-San singular}
        \sup_{ \{u \in W^{1,N}_{0}(\Omega): \int_{\Omega} \lvert \nabla u \rvert^N \, dx \le 1 \} } \int_{\Omega} \frac{e^{\alpha \lvert u \rvert^{N'}}}{\lvert x \rvert^{\sigma}} \, dx < \infty
    \end{align}
    if and only if $\alpha \le \omega_{N-1}^{1/(N-1)} (N- \sigma).$
    \par Let $0<\beta<1,$ We now introduce another log weight:
    \begin{align*}
        \Tilde{w}_{\beta}(x) = \left(\ln \left(\frac{1}{\lvert x \rvert}\right)\right)^{\beta(N-1)} \quad \text{for} \quad 0< \lvert x \rvert < 1.
    \end{align*}
    This weight was first introduced in \cite{RufClanchi2015,RufN}. In \cite{adimurthi2026leraytrudingertypeexponentialintegrability}, the following sharp radial Leray-Trudinger inequalities (see \cite[Theorem $1.7,$ Proposition $1.2$]{adimurthi2026leraytrudingertypeexponentialintegrability}) were established:
        For $0<\beta<1$, there holds,
     \begin{equation} \label{equation sharp radial Leray moser beta}
 \sup_{\left \lbrace u \in W^{1,N}_{0,rad}\left(\mathbb{B}_{N}, \Tilde{w}_{\beta}\right) :\left \lVert \nabla u \right\rVert_{N, \Tilde{w}_{\beta}} \le 1\right\rbrace}  \fint_{\mathbb{B}_N} e^{\alpha \lvert u \rvert^{N'}  X_1^{-\beta}} \, dx <\infty,
    \end{equation} iff $\alpha \leq \alpha_{N, \beta}\coloneqq N\omega_{N-1}^{\frac{1}{N-1}} (1-\beta) $; while for $\beta=1,$ we have,
     \begin{equation} \label{equation sharp radial Leray moser}
\sup_{\left \lbrace u \in W^{1,N}_{0,rad}\left(\mathbb{B}_{N}, w_{1}\right) :\left \lVert \nabla u \right\rVert_{N, w_{1}} \le 1\right\rbrace}  \fint_{\mathbb{B}_N} e^{\alpha \lvert u \rvert^{N'} X_2 X_1^{-1}} \, dx <\infty,
    \end{equation} iff $\alpha \leq \alpha_{N} \coloneqq N \omega_{N-1}^{1/(N-1)}$, where $w_1$ is defined as in \eqref{weights w}.
    In this article, we establish an interpolation between Hardy-type inequalities and Leray-Trudinger-type inequalities \eqref{equation sharp radial Leray moser}, \eqref{equation sharp radial Leray moser beta}. In the spirit of \eqref{Adi-San singular} we prove the following results. 
\begin{proposition} \label{moser weighted1}
     Let $N\geq 2$. Then, 
     \begin{equation} \label{equation sharp radial moser}
\sup_{\left \lbrace u \in W^{1,N}_{0,rad}\left(\mathbb{B}_{N}, w_{1}\right) :\left \lVert \nabla u \right\rVert_{N, w_{1}} \le 1\right\rbrace}  \fint_{\mathbb{B}_N} \frac{e^{\alpha \lvert u \rvert^{N'} X_2 X_1^{-1}}}{\lvert x \rvert^{\sigma}} \, dx <\infty,
    \end{equation} iff $\alpha/\alpha_N + \sigma/N  \le 1$, where $\alpha>0$ and $\sigma \in [0,N)$. Moreover, for $ \alpha/\alpha_N + \sigma/N \le 1$ the supremum is attained in $W^{1,N}_{0,rad}\left(\mathbb{B}_{N}, w_{1}\right)$.
\end{proposition}
\begin{theorem} \label{moser weighted3}
    Let $N\geq 2$ and $0<\beta <1$. Then 
     \begin{equation} \label{equation sharp radial moser beta}
  \sup_{\left \lbrace u \in W^{1,N}_{0,rad}\left(\mathbb{B}_{N}, \Tilde{w}_{\beta}\right) :\left \lVert \nabla u \right\rVert_{N, \Tilde{w}_{\beta}} \le 1\right\rbrace}  \fint_{\mathbb{B}_N} \frac{e^{\alpha \lvert u \rvert^{N'}  X_1^{-\beta}}}{\lvert x \rvert^{\sigma}} \, dx <\infty,
    \end{equation} iff $\alpha/\alpha_{N, \beta} + \sigma/N \le 1 $, where $\alpha>0$ and $\sigma \in [0,N)$.
\end{theorem}
For the non-radial functions, a non sharp version of Proposition \ref{moser weighted1} and \Cref{moser weighted3} just follows from the H\"older's inequality. Let us now state a Lemma which will be useful later, for the proof one can refer to Lemma $5$ of \cite{RufClanchi2015, RufN}.
\begin{lemma}\label{radial lemma}
    Let  $u \in C^1_{c,\mathrm{rad}}(\mathbb{B}_N)$. Then, the following point wise estimates hold.
        \begin{align*}
        \lvert u(x) \rvert \le \frac{1}{\omega_{N-1}^{1/N}} \left(\ln \ln \frac{e}{\lvert x \rvert}\right)^{1/N'} \left(\int_{\mathbb{B}_N} \lvert \nabla u \rvert^N X_1^{-N+1} \, dx \right)^{1/N}.
  \end{align*} 
\end{lemma}
We also need another Lemma, see \cite[Lemma $2.1$]{adimurthi2026leraytrudingertypeexponentialintegrability} for a proof.
\begin{lemma}\label{contains lip functions next}
    Let, $0<\beta <\infty$. Then  $ W^{1,N}_{0}\left(\mathbb{B}_N, \Tilde{w}_{\beta}\right)$ contains all locally Lipschitz functions vanishing at the boundary. 
\end{lemma}
\begin{proof} [Proof of Proposition \ref{moser weighted1}]
    Let us first assume that $\alpha/\alpha_{N} + \sigma/N < 1$. So, there exists $t>1$ such that $\alpha/\alpha_{N} + (\sigma t)/N = 1$. Thus, in this case, \eqref{equation sharp radial moser} follows by Holder's inequality together with \eqref{equation sharp radial Leray moser}, indeed we have
    \begin{align*}
        \int_{\mathbb{B}_N} \frac{e^{\alpha_N \lvert u \rvert^{N'} X_2 X_1^{-1}}}{\lvert x \rvert^{\sigma}} \, dx \le  \left(\int_{\mathbb{B}_N} e^{\alpha_N \lvert u \rvert^{N'} X_2 X_1^{-1}} \, dx\right)^{\frac{\alpha}{\alpha_N}} \left(\int_{\mathbb{B}_N} \frac{ dx}{\lvert x \rvert^{N/t}}\right)^{\frac{\sigma t}{N}}.
    \end{align*}
  \par  Now, consider $\alpha/\alpha_{N} + \sigma/N = 1$. Let $a=\alpha/\alpha_N$ so, $1-a = \sigma/N$. Assume that $u \in C^1_{c, rad}\left(\mathbb{B}_N\right).$ Using Lemma \ref{radial lemma}, we can deduce that,
  \begin{align*}
          \int_{\mathbb{B}_N} \frac{e^{\alpha \lvert u \rvert^{N'} X_2 X_1^{-1}}}{\lvert x \rvert^{\sigma}} \, dx &= \omega_{N-1} \int_{0}^{1} e^{\alpha \lvert u \rvert^{N'}(r) X_2(r) X_1^{-1}(r)} r^{N-1- \sigma} \, dr \\
    &  \le \omega_{N-1} \int_{0}^{1} e^{a N \frac{\ln \ln\frac{e}{r}}{1 + \ln \ln\frac{e}{r}}\ln\frac{e}{r}} r^{aN-1} \, dr\\
  &=\omega_{N-1} e^{aN} \int_{1}^{\infty} e^{-\frac{aNz}{1+ \ln z}} \, dz\le C_N.
     \end{align*}
\par Finally, assume that $\alpha/\alpha_{N} + \sigma/N > 1.$ For $0<r_1<1$, we define
    \begin{align}\label{radial ltm function}
       \eta_{r_1}(x) \coloneqq \begin{cases}
        \frac{1}{\omega_{N-1}^{1/N}} \left(\ln \ln \frac{e}{r_1}\right)^{1/N'}, \quad 0 \le \lvert x \rvert < r_1\\
        \frac{1}{\omega_{N-1}^{1/N}} \frac{\ln \ln \frac{e}{\lvert x \rvert}}{\left(\ln \ln \frac{e}{r_1}\right)^{1/N}}, \quad r_1 \le \lvert x \rvert \le 1.
    \end{cases}
    \end{align} 
It can be checked that $\int_{\mathbb{B}_N} \lvert \nabla \eta_{r_1} \rvert^N X_1^{-N+1} =1$ and from Lemma \ref{contains lip functions}, it follows that $\eta_{r_1} \in W^{1,N}_0\left(\mathbb{B}_N, X_1^{-N+1}\right)$. Note that we have,
\begin{align}\label{moser prop 1}
\frac{1}{\omega_{N-1}}\int_{\mathbb{B}_N} \frac{e^{\alpha \left\lvert \eta_{r_1} \right\rvert^{N'} X_2X_1^{-1}}}{\lvert x \rvert^{\sigma}}\geq   \int_{0}^{r_1} e^{\alpha \eta_{r_1}^{N'}(r) X_2(r) X_1^{-1}(r)} r^{N-\sigma-1} \, dr \eqqcolon I_{r_1}.
\end{align}
Since $X_2 X_1^{-1}$ is a decreasing function in $(0,1)$, so we have
$$X_2(r) X_1^{-1}(r) \ge \left(\ln \frac{e}{r_1}\right)\left(1 + \ln \ln \frac{e}{r_1}\right)^{-1} \eqqcolon f(r_1), \text{ for all } r \in (0,r_1).$$
Thus, we have,
\begin{align*}
    I_{r_1}\ge \int_{0}^{r_1} e^{ \alpha \omega_{N-1} ^ { - \frac{1}{N-1}} \ln \frac{e}{r_1}f(r_1)} r^{N-\sigma-1} \, dr= \frac{e^{N-\sigma}}{N-\sigma} e^{N  \left(\frac{\alpha}{\alpha_N}f(r_1) + \frac{\sigma}{N}-1\right) \ln \frac{e}{r_1}} \to \infty,
\end{align*}
as $r_1 \to 0$ since $\lim_{r_1 \to 0} f(r_1) =1$.
\par The attainability of the supremum follows from Lemma \ref{radial lemma} together with the dominated convergence Theorem, we omit the details.
\end{proof}
\begin{proof} [Proof of \Cref{moser weighted3}]
Assume that $ u \in C^1_{c, rad}\left(\mathbb{B}_N\right).$ Let us first deal with the case $\Tilde{w}_{\beta}$. As in the proof of Proposition \ref{moser weighted1}, if $\alpha/\alpha_{N, \beta} +\sigma/N < 1,$ then the boundedness of the integral follows by the Holder's inequality.
  \par Now, assume that $\alpha/\alpha_{N, \beta} + \sigma/N = 1.$ Write $a \coloneqq \alpha/\alpha_{N, \beta},$ so $\sigma/N = 1-a.$ Now, set $v(r) = a^{\frac{1-\beta}{N'}} u\left(r^{1/a}\right)$. Then $v'(r) = a^{\frac{1-\beta}{N'}-1} u'(r^{1/a})r^{1/a-1}$ and
     \begin{align*}
         \int_{0}^{1} \lvert v'(r) \rvert^{N} \left(\ln \frac{1}{r}\right)^{\beta (N-1)} r^{N-1} \, dr=\int_{0}^1 \lvert u'(s) \rvert^N \left(\ln \frac{1}{s}\right)^{\beta (N-1)} s^{N-1} \, ds.
     \end{align*}
We also have,
\begin{align*}
 \int_{0}^{1} e^{\alpha_{N,\beta} \lvert v(r) \rvert^{N'} \left(\ln \frac{e}{r}\right)^{\beta}} r^{N-1} \, dr &= \int_{0}^1 e^{\alpha_{N,\beta} a^{1-\beta}\lvert u(r^{1/a}) \rvert^{N'} \left(\ln \frac{e}{r}\right)^{\beta}} r^{N-1} \, dr\\
 &= a\int_{0}^1 e^{a\alpha_{N,\beta} \lvert u(s) \rvert^{N'} \left(\ln \frac{e^{1/a}}{s}\right)^{\beta}} s^{Na-1} \, ds   \\
 &\ge  a\int_{0}^{1} e^{a \alpha_{N,\beta} \lvert u(r) \rvert^{N'}  X_1^{-\beta}} r^{aN-1} \, dr .
\end{align*}
Therefore we have,
\begin{align*}
   \sup_{\left \lbrace \left \lVert \nabla u \right\rVert_{N, \Tilde{w}_{\beta}} \le 1\right \rbrace} \int_{\mathbb{B}_N} \frac{e^{\alpha \lvert u \rvert^{N'}  X_1^{-\beta}}}{\lvert x \rvert^{\sigma}} \, dx \stackrel{\eqref{equation sharp radial Leray moser beta}}{\le} \frac{1}{a}\sup_{\left \lbrace \left \lVert \nabla v \right\rVert_{N, \Tilde{w}_{\beta}} \le 1\right \rbrace}\int_{\mathbb{B}_N} e^{\alpha_{N,\beta} \lvert v \rvert^{N'} X_1^{-\beta}} \, dx < \infty.
\end{align*}
\par Finally assume that $\alpha/\alpha_{N, \beta} + \sigma/N> 1.$ For $\varepsilon \in (0,1)$, we define,
\begin{align} \label{broken-line}
        \xi_{\varepsilon}(r) \coloneqq \begin{cases}
            \frac{1}{\omega_{N-1}^{\frac{1}{N}}(1-\beta)^{\frac{1}{N'}}} \left(\ln \frac{1}{\varepsilon}\right)^{\frac{1-\beta}{N'}} \quad 0 \le r < \varepsilon\\
            \frac{1}{\omega_{N-1}^{\frac{1}{N}}(1-\beta)^{\frac{1}{N'}}} \left[\frac{\ln \frac{1}{r}}{\left(\ln \frac{1}{\varepsilon}\right)^{1/N}}\right]^{1-\beta} \quad \varepsilon \le r <1.
        \end{cases} 
    \end{align}
From Lemma \ref{contains lip functions next}, it follows that $\xi_{\varepsilon} \in W^{1,N}_{0}\left(\mathbb{B}_N, \Tilde{w}_{\beta}\right),$ and by simple computations, it can be checked that $\left \lVert \nabla \xi_{\varepsilon} \right\rVert_{N, \Tilde{w}_{\beta}} =1.$
Now,
 \begin{align*}
         \int_{0}^1 e^{\alpha \lvert \xi_{\varepsilon} \rvert ^{N'} \left(\ln \frac{1}{r}\right)^{\beta}}r^{N-\sigma-1} \, dr
         &\ge \int_{0}^{\varepsilon} \exp \left(\frac{N \alpha }{\alpha_{N,\beta}} \left(\ln \frac{1}{\varepsilon}\right)^{1-\beta} \left(\ln \frac{1}{r}\right)^{\beta}\right) r^{N-\sigma-1} \, dr\\
        &\ge \int_{0}^{\varepsilon} e^{ \frac{N \alpha }{\alpha_{N,\beta}}\ln \frac{1}{\varepsilon} } r^{N-\sigma-1} \, dr= \frac{1}{N-\sigma} \varepsilon^{N- \sigma - \frac{N \alpha }{\alpha_{N,\beta}}} \to \infty,
    \end{align*}
    $\varepsilon \to 0$. Now, consider the weight $w_{\beta}.$ Since $\left \lVert \nabla u \right\rVert_{N, \Tilde{w}_{\beta}} \le \left \lVert \nabla u \right\rVert_{N, w_{\beta}},$ so the boundedness of the integral follows in the sub-critical and critical case. We will check the super-critical case. For $0 < r_1 < 1,$ we consider the family of functions given by,
 \begin{align}
        u_{r_1}(r) \coloneqq \begin{cases}
            \frac{1}{\omega_{N-1}^{1/N}} \frac{\left[ \left(\ln \frac{e}{r_1}\right)^{1-\beta}-1\right]^{1/N'}}{\left(1-\beta\right)^{1/N'}} , \quad 0 \le r < r_1\\
            \frac{1}{\omega_{N-1}^{1/N}} \frac{ \left(\ln \frac{e}{r}\right)^{1-\beta}-1}{\left(1-\beta\right)^{1/N'}\left[\left(\ln \frac{e}{r_1}\right)^{1-\beta}-1\right]^{1/N}} , \quad r_1 \le r < 1.
        \end{cases}
    \end{align}
It can be checked that, $\| u_{r_1} \|_{N,w_{\beta}} = 1$ and from Lemma \ref{contains lip functions}, it follows that $u_{r_1} \in W^{1,N}_0 \left (\mathbb{B}_N, w_{\beta}\right)$.
Therefore we have,
\begin{align*}
    I &\coloneqq \omega_{N-1} \int_{0}^1 e^{\alpha \lvert u_{r_1} \rvert ^{N'} \left(\ln \frac{e}{r}\right)^{\beta}} r^{N-\sigma-1} \, dr\\
    &\ge \int_{0}^{r_1} \exp \left(\frac{\alpha}{\omega_{N-1}^{\frac{1}{N-1}}(1-\beta)} \left[ \left( \ln \frac{e}{r_1}\right)^{1-\beta}-1\right]\left( \ln \frac{e}{r}\right)^{\beta}\right) r^{N-\sigma-1} \, dr\\
    &\ge \int_{0}^{r_1} \exp \left(\frac{N\alpha}{\alpha_{N,\beta}} \left[ \left( \ln \frac{e}{r_1}\right)^{1-\beta}-1\right]\left( \ln \frac{e}{r_1}\right)^{\beta}\right) r^{N-\sigma-1} \, dr\\
    &= \frac{1}{N- \sigma}  e^{\frac{N\alpha}{\alpha_{N,\beta}}\left[ \ln \frac{e}{r_1}- \left(\ln \frac{e}{r_1}\right)^{\beta}\right]} r_1^{N-\sigma} \, dr\\
   &=c_{N,\sigma,\beta} e^{N\ln \frac{e}{r_1} \left(\frac{\alpha}{\alpha_{N,\beta}} -1 + \frac{\sigma}{N} + \frac{\alpha}{\alpha_{N,\beta}}\left(\ln \frac{e}{r_1}\right)^{\beta-1}\right)} \to \infty \quad \text{as} \quad r_1 \to 0.
\end{align*}
This completes the proof.
\end{proof}

\section*{Acknowledgments}
The author thanks \emph{Prof.} Arka Mallick for valuable suggestions and discussions. The author is also thankful to \emph{Prof.} Adimurthi for suggesting the interpolation of Hardy-type and the Leray-Trudinger-type inequalities. The author is supported by the UGC NET fellowship Govt. of India. The Department of Mathematics, Indian Institute of Science, where this research was carried out, received support of DST FIST program-2021 [TPN-700661].

\end{document}